\documentclass[preprint]{imsart}

\RequirePackage[OT1]{fontenc}
\RequirePackage{amsthm,amsmath}
\RequirePackage[numbers]{natbib}
\RequirePackage[colorlinks,citecolor=blue,urlcolor=blue]{hyperref}

\startlocaldefs
\numberwithin{equation}{section}
\theoremstyle{plain}
\newtheorem{thm}{Theorem}[section]
\newtheorem{prob}{Problem}[section]
\newtheorem{cor}[thm]{Corollary}
\newtheorem{lem}[thm]{Lemma}
\newtheorem{ex}[thm]{Example}

\theoremstyle{definition}
\newtheorem{defn}[thm]{Definition}
\theoremstyle{remark}
\newtheorem{rem}[thm]{Remark}

\endlocaldefs

\begin{document}

\begin{frontmatter}
\title{On objective and strong objective consistent estimates of unknown parameters  for statistical structures in a Polish
group admitting an invariant metric}%\thanksref{T1}}
\runtitle{On objective and strong objective consistent estimates}
%\thankstext{T1}{Footnote to the title with the `thankstext' command.}

\begin{aug}
\author{\fnms{ Murman } \snm{Kintsurashvili}\thanksref{t1}\ead[label=e1]{m.kintsurashvili@gtu.ge}}

\author{\fnms{ Tengiz} \snm{Kiria}\thanksref{t2}\ead[label=e2]{t.kiria@gtu.ge}}
\and
\author{\fnms{Gogi} \snm{Pantsulaia}\thanksref{t3}\ead[label=e3]{g.pantsulaia@gtu.ge}}

\affiliation{Department of Mathematics, Georgian Technical University }%\thanksmark{m1}}

\thankstext{t1}{Phd student,Department of Mathematics, Georgian Technical University}
\thankstext{t2}{Phd student,Department of Mathematics, Georgian Technical University}
\thankstext{t3}{Full Professor,Department of Mathematics, Georgian Technical University}
\runauthor{M.Kintsurashvili, T.Kiria, G.Pantsulaia}

\address{Department of mathematics, Georgian Technical University,\\ Kostava Street. 77 , Tbilisi
DC  0175,\\ Republic of  Georgia \\
\printead{e1,e2,e3}}
\phantom{E-mail:\ }\printead*{}
\end{aug}

\begin{abstract}By using the notion of a Haar ambivalent set introduced by Balka, Buczolich and Elekes (2012), essentially new  classes of statistical structures having objective and strong objective estimates of unknown parameters are introduced in a Polish
non-locally-compact group admitting an invariant metric and relations between them  are studied in this paper. An example of such a weakly separated statistical structure  is constructed  for which a question asking ''{\it whether there exists a consistent estimate of an unknown parameter}''  is not solvable within the theory $(ZF)~\&~(DC)$.  A question asking ''{\it whether there exists an objective  consistent estimate of an unknown  parameter for any statistical structure in a non-locally compact Polish group with an invariant metric when subjective one exists}''  is answered positively when there exists  at least one  such a parameter the pre-image of which under this subjective estimate is a prevalent.
These results extend recent results of authors.  Some examples  of objective and strong objective consistent estimates in a
compact Polish group $\{0; 1\}^N$ are considered in this paper.

\end{abstract}


%\tableofcontents
\end{frontmatter}

\section{Introduction} In order to explain  a big gap between the
theory of mathematical statistics and results of hypothesis testing,
 concepts of subjective and
objective infinite sample consistent estimates of a useful
signal in the linear one-dimensional stochastic model were
introduced in \cite{Pan14}. This approach essentially uses the concept of  Haar null sets in Polish topological vector spaces  introduced by
J.P.R. Christensen \cite{Chris73}.

The Polish topological vector space  ${\bf  R}^N$ of all real-valued sequences (equivalently, of infinite samples) equipped with Tychonoff metric plays a central role  in the theory of statistical decisions because
 a definition of any consistent estimate of an unknown parameter in various stochastic models without infinite samples is simply impossible.

Let explain from the point of view of the theory of Haar null sets in ${\bf  R}^N$  some confusions which were described
 by Jum Nunnally \cite{ Nunnally1960} and Jacob Cohen \cite{Cohen1994}:

 Let  $x_1,x_2, \cdots$  be an infinite sample obtained by  observation on  independent and  normally distributed real-valued random
variables with  parameters $(\theta,1)$, where $\theta$ is an unknown mean and the variance is equal to $1$.  Using this  infinite  sample we want to estimate
an unknown mean. If we denote by $\mu_{\theta}$ a linear Gaussian measure on ${\bf  R}$ with the probability density $\frac{1}{\sqrt{2\pi}}e^{-\frac{(x-\theta)^2}{2}}$, then the triplet
\begin{equation}
({\bf  R}^N,\mathcal{B}({\bf  R}^N),\mu_{\theta}^N)_{\theta \in R}\label{ccs}
\end{equation}
stands a statistical structure described our experiment, where $\mathcal{B}({\bf  R}^N)$ denotes the $\sigma$-algebra of Borel subsets of ${\bf  R}^N$. By virtue of the Strong Law of Large Numbers we know that the condition
\begin{equation}
 \mu_{\theta}^N(\{(x_k)_{k \in N}: (x_k)_{k \in N}\in {\bf  R}^N~\&~\lim_{n \to \infty}\frac{\sum_{k=1}^nx_k}{n}=\theta\}=1 \label{ccs}
\end{equation}
 holds true for each $\theta \in {\bf  R}$.

Take into account the validity of (1.2),  for construction of  a consistent infinite sample estimation of an unknown parameter $\theta$ a mapping $T$ defined by
\begin{equation}
T((x_k)_{k \in N})=\lim_{n \to \infty}\frac{\sum_{k=1}^nx_k}{n}\label{ccs}
\end{equation}
is used  by statisticians. As usual, null
hypothesis significance testing in the case $H_0:\theta=\theta_0$ assumes the following procedure: if an infinite sample $(x_k)_{k \in N} \in T^{-1}(\theta_0)$ then $H_0$ hypothesis is accepted and  $H_0$ hypothesis is rejected, otherwise. There naturally arises a question asking  whether can be explained  Jacob Cohen statement \cite{Cohen1994}: {\it "... Don’t look for a magic alternative to NHST [null
hypothesis significance testing] ... It does not exist."} Notice that a set $S$ of all infinite samples $(x_k)_{k \in N}$ for which there exist  finite limits of arithmetic means of their first $n$ elements  constitutes a proper Borel measurable vector subspace of ${\bf  R}^N$. Following  Christensen \cite{Chris73}, each proper Borel measurable vector subspace of an arbitrary Polish topological vector space is Haar null set and since  $S$ is a Borel measurable proper vector subspace of ${\bf  R}^N$ we claim  that the mapping $T$ is not defined for "almost every"(in the sense of Christensen \footnote{We say that a sentence $P(\cdot)$ formulated in term of an element of a Polish group $G$ is true  for "almost every" element
of $G$ if a set of all elements $g \in G$ for which $P(g)$ is false constitutes a Haar null set in $G$.} \cite{Chris73}) infinite sample. The latter relation means that for "almost every" infinite sample we reject null hypothesis $H_0$. This discussion can be used also to explain
Jum  Nunnally's \cite{ Nunnally1960} following conjecture: {\it "If the
decisions are based on convention they are termed arbitrary or
mindless while those not so based may be termed subjective. To
minimize type $II$ errors, large samples are recommended. In
psychology practically all null hypotheses are claimed to be false
for sufficiently large samples so ... it is usually nonsensical to
perform an experiment with the sole aim of rejecting the null
hypothesis".}

Now let $T_1:{\bf  R}^N \to R$ be another infinite sample consistent estimate of an unknown parameter $\theta$ in the above mentioned model, i.e.
\begin{equation}
 \mu_{\theta}^N(\{(x_k)_{k \in N}: (x_k)_{k \in N}\in {\bf  R}^N~\&~T_1((x_k)_{k \in N})=\theta\})=1 \label{ccs}
\end{equation}
for each $\theta \in {\bf  R}$.
Here naturally arises a question asking what are  those additional  conditions imposed on the estimate $T_1$  under  which the above-described  confusions  will be settled.

In this direction, first, notice that there must be no a parameter $\theta_0 \in R$ for which $T_1^{-1}(\theta_0)$ is Haar null set, because then for "almost every" infinite sample null hypothesis $H_0: \theta=\theta_0$ will be rejected. Second, there must be no a parameter $\theta_1 \in {\bf  R}$ for which $T_1^{-1}(\theta_1)$ is a prevalent set (equivalently, a  complement of a Haar null set) because then  for "almost every" infinite sample  null hypothesis $H_0: \theta=\theta_2$ will be rejected for each $\theta_2 \neq \theta_1$.
This observations lead us to additional  conditions imposed on the estimate $T_1$  which assumes that  $T_1^{-1}(\theta)$ must be  neither Haar null  nor
prevalent for each $\theta \in {\bf  R}$. Following \cite{Balka12}, a set which is neither Haar null  nor
prevalent is called a Haar ambivalent set.  Such estimates firstly were  adopted as objective  infinite sample consistent estimates of a useful signal in the linear one-dimensional stochastic  model(see, \cite{Pan13}, Theorem 4.1, p. 482).

It was proved in \cite{Pan13}  that
$T_n :{{\bf  R}}^n \to {{\bf  R}}$~$(n \in N)$ defined by
\begin{equation}
 T_n(x_1, \cdots, x_n)=-F^{-1}(n^{-1}\#( \{ x_1, \cdots, x_n \}
\cap (-\infty;0]))\label{ccs}
\end{equation}
for $(x_1, \cdots, x_n) \in {{\bf  R}}^n$,
is a consistent estimator of a useful signal $\theta$ in
one-dimensional  linear stochastic model
\begin{equation}
 \xi_k=\theta+\Delta_k~(k \in N),\label{ccs}
\end{equation}
where $\#(\cdot)$ denotes a counting
measure, $\Delta_k$  is a sequence of independent identically
distributed random variables on ${\bf  R}$ with strictly
increasing continuous distribution function $F$ and expectation of
$\Delta_1$ does not exist. In this direction the following two examples of simulations of  linear one-dimensional  stochastic
models have been considered.

\begin{ex}( \cite{Pan13}, Example 4.1, p. 484) Since a sequence of
real numbers $(\pi \times n -[\pi \times n])_{n \in N}$, where
$[\cdot]$ denotes an integer part of a real number,  is uniformly
distributed on $(0,1)$(see, \cite{KuNi74}, Example 2.1, p.17), we
claim that a simulation of a $\mu_{(\theta,1)}$-equidistributed
sequence $(x_n)_{n \le M}$  on $R$( $M$ is a "sufficiently large"
natural number and depends on a representation quality of the
irrational number  $\pi$), where $\mu_{(\theta,1)}$ denotes a
$\theta$-shift of the measure $\mu$ defined by distribution function $F$,  can
be obtained  by the formula
\begin{equation}
x_n=F_{\theta}^{-1} (\pi \times n -[\pi \times n])\label{ccs}
\end{equation}
for $n \le M$ and $\theta \in R$, where $F_{\theta}$
denotes a  distribution function corresponding to the
measure $\mu_{\theta}$.

 In this  model, $\theta$ stands a "useful signal".

 We set:

 (i) ~$n$ - the number of trials;

(ii) ~$T_n$ - an estimator defined by the formula (1.5);

 (iii) $\overline{X}_n$ - a sample average.

When $F(x)$ is a standard  Gaussian distribution function, by using
Microsoft Excel  we have obtained   numerical data placed in Table 1.
\begin{table*}
\caption{Estimates of the useful signal $\theta=1$ when the white noise is standard  Gaussian random variable}
\label{sphericcase}
\begin{tabular}{crrrrrrc}
\hline
 \multicolumn{1}{c}{$n$} & \multicolumn{1}{c}{$T_n$} & \multicolumn{1}{c}{$\overline{X}_n$} & \multicolumn{1}{c}{$n$} & \multicolumn{1}{c}{$T_n$} & \multicolumn{1}{c}{$\overline{X}_n$} \\
\hline
$ 50 $       &    $0.994457883$  & $1.146952654$  &     $550$ &   $1.04034032$  & $1.034899747$  \\

$ 100 $       &    $1.036433389$  & $1.010190601$  &      $600$ &   $1.036433389$  & $1.043940988$ \\

$ 150 $       &    $1.022241387$  & $1.064790041$  &    $650$ &   $1.03313984$  & $1.036321771$ \\

$ 200 $       &    $1.036433389$  & $1.037987511$  &      $700$ &   $1.030325691$  & $1.037905202$ \\

$ 250 $       &    $1.027893346$  & $1.045296447$  &      $750$ &   $1.033578332$  & $1.03728633$ \\

$ 300 $       &    $1.036433389$  & $1.044049728$  &      $800$ &   $1.03108705$  & $ 1.032630945$  \\

$ 350 $       &    $1.030325691$  & $1.034339407$  &      $850$ &   $1.033913784$  & $1.037321098$ \\

$ 400 $       &    $1.036433389$  & $1.045181911$  &      $900$ &   $1.031679632$  & $ 1.026202323$ \\

$ 450 $       &    $1.031679632$  & $1.023083495$  &      $950$ &   $1.034178696$  & $1.036669278$  \\

$ 500 $       &    $1.036433389$  & $1.044635371$  &      $1000$ &   $1.036433389 $  & $1.031131694
$  \\
\end{tabular}
\end{table*}
Notice  that results of computations presented in Table 1 show us that both statistics $T_n$
and $\overline{X}_n$ give us a good estimates of the "useful signal" $\theta$ whenever a generalized "white noise"
in that case has a finite absolute moment
of the first order  and its moment of the first order is equal to
zero.

Now let $F$ be  a
linear Cauchy distribution function  on $R$, i.e.
\begin{equation}
F(x)=\int_{-\infty}^x\frac{1}{\pi(1+t^2)}dt ~(x \in R). \label{ccs}
\end{equation}

Numerical data  placed in Table 2 were obtaining by using Microsoft Excel  and Cauchy distribution
calculator of the high accuracy \cite{Kaisan}.
\begin{table*}
\caption{Estimates of the useful signal $\theta=1$ when the white noise is Cauchy random variable}
\label{sphericcase}
\begin{tabular}{crrrrc}
\hline
 \multicolumn{1}{c}{$n$} & \multicolumn{1}{c}{$T_n$} & \multicolumn{1}{c}{$\overline{X}_n$}  & \multicolumn{1}{c}{$n$} & \multicolumn{1}{c}{$T_n$} & \multicolumn{1}{c}{$\overline{X}_n$} \\
\hline
$ 50 $       &    $1.20879235$  & $2.555449288$   &     $550$ &   $1.017284476$  & $41.08688757$  \\

$ 100 $       &    $0.939062506$  & $1.331789564$   &     $600$ &   $1.042790358$  & $ 41.30221291$ \\

$ 150 $       &    $1.06489184$  & $71.87525566$   &     $650$ &   $1.014605804$  & $38.1800532$  \\

$ 200 $       &    $1.00000000$  & $54.09578271$   &     $700$ &   $1.027297114$  & $38.03399768$  \\

$ 250 $       &    $1.06489184$  & $64.59240343$   &     $750$ &   $1.012645994$  & $35.57956117$  \\

$ 300 $       &    $1.021166379$  & $54.03265563$   &     $800$ &   $1.015832638$  & $ 35.25149408$  \\

$ 350 $       &    $1.027297114$  & $56.39846672$   &     $850$ &   $1.018652839$  & $33.28723503$  \\

 $ 400 $       &    $1.031919949$  & $49.58316089$   &     $900$ &   $1.0070058$  & $31.4036155$  \\

$ 450 $       &    $1.0070058$  & $ 44.00842613$   &     $950$ &   $1.023420701$  & $31.27321466$  \\

$ 500 $       &    $1.038428014$  & $ 45.14322051$   &     $1000$ &   $1.012645994$  & $29.73405416$  \\
\end{tabular}
\end{table*}
On the one hand, the results of computations  placed in Table
2 do not contradict to the above mentioned fact
that $T_n$ is a consistent estimator of the parameter $\theta=1$.
On the other hand, we know that  a sample average $\overline{X}_n$
does not work in that case  because  the mean and variance of the "white noise "(i.e., Cauchy random variable)
are not defined. By this reason  attempts to estimate the "useful
signal" $\theta=1$ by using the sample average  will not be
successful.
\end{ex}
In \cite{Pan13} has  been established that
the estimators $
\overline{\lim}\widetilde{T_n} := \inf_n \sup_{m \ge
n}\widetilde{T_m}$ and $ \underline{\lim}\widetilde{T_n} := \sup_n
\inf_{m \ge n}\widetilde{T_m}$ are
consistent infinite sample estimates of a useful signal $\theta$ in the model (1.6) (see, \cite{Pan13}, Theorem 4.2, p. 483). When we begin to study properties of these infinite sample estimators from the point of view of the theory of Haar null sets in ${\bf  R}^N$, we observed a surprising and an unexpected fact for us that these  both estimates
are objective (see, \cite{Pan15}, Theorem 3.1).

As the described approach naturally divides  a class of consistent infinite sample estimates into  objective and subjective estimates
shouldn't seem excessively highly told our suggestion
that  each consistent   infinite sample estimate must pass
the theoretical test  on  the objectivity.

The present manuscript introduces the concepts of the theory of objective infinite sample  consistent estimates in ${\bf R}^N$ and  gives its extension to all non-locally-compact Polish groups admitting an invariant metric.

The rest of this note is the following.

In Section 2 we give some notions and facts from the theory of Haar null sets in complete metric linear spaces and equidistributed sequences on the real axis ${\bf R}$. Concepts of objective and strong objective infinite sample consistent  estimates for statistical structures are introduced also in this section.
Section 3 presents a certain construction of the objective infinite sample  consistent estimate of an unknown distribution function which generalises the recent results  obtained in \cite{Pan13}. There is proved  an existence of the infinite sample  consistent estimate of an unknown distribution function $F (F \in \mathcal{F})$ for the family of Borel probability measures $\{ p_F^N :  F \in \mathcal{F}\}$, where $\mathcal{F}$ denotes the family of all strictly increasing and continuous
distribution functions on ${\bf R}$ and $p_F^N$ denotes an infinite power of the  Borel probability measure $p_F$ on ${\bf R}$
defined  by $F$. Section 4 presents an effective construction of the  strong objective  infinite sample consistent estimate of the "useful
signal" in a certain linear one-dimensional stochastic model. An infinite  sample consistent estimate of an unknown probability density  is constructed for the
separated class of positive continuous probability densities and a problem about existence of an objective one is stated in Section 5.
In Section 6,  by using the notion of a Haar ambivalent set introduced in \cite{Balka12}, essentially new  classes of statistical structures having objective and strong objective estimates of an unknown parameter are introduced in a Polish
non-locally-compact group admitting an invariant metric and relations between them  are studied in this section. An example of such a weakly separated statistical structure  is constructed  for which a question asking {\it whether there exists a consistent estimate of an unknown parameter}  is not solvable within the theory $(ZF)~\&~(DC)$.  These results extend recent results obtained in \cite{Pan14-1}. In addition, we extend the concept of objective and subjective  consistent estimates introduced for ${\bf  R}^N$ to all Polish groups  and consider a question asking whether there exists an objective  consistent estimate of an unknown  parameter for any statistical structure in a non-locally compact Polish group with an invariant metric when subjective one exists. We show that this question is answered positively when there exists  at least one  such a parameter the pre-image of which under this  subjective estimate is a prevalent. In Section 7 we consider some examples  of objective and strong objective consistent estimates in a compact Polish
group $\{0; 1\}^N$.

\section{ Auxiliary notions and facts from functional analysis and measure theory}

Let ${V}$ be a
complete metric linear space, by which we mean a vector space
(real or complex) with a complete metric for which the operations
of addition and scalar multiplication are continuous. When we
speak of a measure on ${V}$ we will always mean a
nonnegative measure that is defined  on the Borel sets of
${V} $ and is not identically zero. We write $S + v$ for
the translation of a set $S \subseteq {V} $ by a vector $v
\in {V} $.

%vspace*{0.3cm}{\bf  Definition 2.1}
\begin{defn}
( \cite{HSY92}, Definition 1, p. 221) A
measure $\mu$ is said to be transverse to a Borel
 set $S \subset {V} $ if the
following two conditions hold:

(i)~There exists a compact set $U \subset {V} $ for which
$0 < \mu(U) <1$;

(ii)~ $\mu(S + v) = 0$ for every $v \in {V} $.

\end{defn}

%\vspace*{0.3cm}{\bf  Definition 2.2}
\begin{defn} (\cite{HSY92}, Definition 2, p. 222;
\cite{Balka12}, p. 1579 ) A Borel set $S \subset {V} $ is called shy
if there exists a measure transverse to $S$. More generally, a
subset of ${V} $ is called shy if it is contained in a shy
Borel set. The complement of a shy set is called a prevalent set.
 We say that a set is Haar ambivalent if
it is neither shy nor prevalent.
\end{defn}

%\vspace*{0.3cm}{\bf  Definition 2.3}
\begin{defn}(\cite{HSY92}, p. 226) We say that
''almost every'' element of ${V}$ satisfies some given
property, if the subset of ${V}$ on which this property
holds is prevalent.
\end{defn}

\begin{lem}[ \cite{HSY92}, Fact $3^{''}$,  p. 223] The union of a countable collection of shy sets is shy.~\end{lem}

\begin{lem}[ \cite{HSY92}, Fact $8$,  p. 224]
If ${V}$  is infinite dimensional, all compact subsets of
${V} $ are shy.~\end{lem}

\begin{lem}[\cite{Khar84}, Lemma 2,  p. 58] Let $\mu$ be a Borel
 probability measure defined in complete separable metric space ${V} $.
 Then there exists a countable family
 of compact sets $(F_k)_{k \in N}$ in ${V} $ such that
 $\mu({V} \setminus \cup_{k \in N}F_k)=0.$
\end{lem}

Let ${\bf  R}^N$  be a  topological vector space of all real
valued sequences  equipped with Tychonoff
metric $\rho$ defined by $ \rho((x_k)_{k \in N}, (y_k)_{k
\in N})=\sum_{k \in N}|x_k-y_k|/2^k(1+|x_k-y_k|)
$ for $(x_k)_{k \in N}, (y_k)_{k \in N} \in
{\bf  R}^N$.

%\vspace*{0.3cm}{\bf  Lemma 2.1}
\begin{lem}( \cite{Pan07}, Lemma 15.1.3, p. 202  ) Let $J$
be an arbitrary   subset of $N$. We set
\begin{equation}
A_J=\{(x_i)_{i \in N}: x_i \le 0~\mbox{for}~i \in
J~\&~x_i > 0~\mbox{for}~i \in N \setminus J\}.\label{ccs}
\end{equation}

Then the family of subsets ~$\Phi=\{A_J:J \subseteq N \}$
has the following properties:

$(i)$ every element of $\Phi$  is Haar ambivalent.

$(ii)$ ~$A_{J_1} \cap A_{J_2}=\emptyset$ for all different
$J_1,J_2 \subseteq N$.

$(iii)$ ~$\Phi$ is a partition of ${\bf  R}^N$
 such that $\mbox{card}(\Phi)=2^{\aleph_0}$.
\end{lem}

%\vspace*{0.3cm}{\bf  Remark 2.1}

\begin{rem}
The  proof of the Lemma 2.7 employs an argument stated that each Borel subset of  ${\bf  R}^N$ which for each compact set contains it's any translate is non-shy set.
\end{rem}

%\vspace*{0.3cm}{\bf  Definition 2.4}

%\vspace*{0.3cm}{\bf  Definition 2.4}

%\vspace{.08in} \noindent{\bf  Definition 2.1

\begin{defn}(\cite{KuNi74}) A sequence $(x_k)_{k \in N}$ of real numbers from the interval
$(a, b)$ is said to be equidistributed or uniformly distributed on
an interval $(a, b)$ if for any subinterval $[c, d]$ of  $(a, b)$
we have
\begin{equation}
\lim_{n \to \infty}
n^{-1}\#(\{x_1, x_2, \cdots, x_n\} \cap [c,d])=(b-a)^{-1}(d-c),\label{ccs}
\end{equation} where $\#$ denotes a counting measure.
\end{defn}

Now let $X$ be a compact Polish space and $\mu$ be a probability
Borel measure on $X$. Let $\mathcal{R}(X)$ be a space of all
bounded continuous functions defined on  $X$.

%\vspace{.08in} \noindent {\bf  Definition 2.2}
\begin{defn}A sequence $(x_k)_{k \in N}$ of elements of $X$ is said to be $\mu$-equidistri- \newline
butted or $\mu$-uniformly distributed on the $X$
 if for every $f \in \mathcal{R}(X)$ we have
\begin{equation}
\lim_{n \to \infty}n^{-1}\sum_{k=1}^nf(x_k) =\int_{X}fd\mu.\label{ccs}
\end{equation}
\end{defn}

 %\noindent{\bf  Lemma 2.2}
 \begin{lem} (\cite{KuNi74}, Lemma 2.1, p.
199) Let  $f \in \mathcal{R}(X)$. Then, for  $\mu^N$-almost
 every sequences $(x_k)_{k \in N} \in X^{N}$,  we have
\begin{equation}
\lim_{n \to \infty} n^{-1}\sum_{k=1}^n f(x_k)=\int_{X}f d \mu.  \label{ccs}
\end{equation}
\end{lem}

 %\noindent{\bf  Lemma 2.3}

 \begin{lem} (\cite{KuNi74}, pp. 199-201)
Let $S$ be a set of all $\mu$-equidistributed sequences on $X$.
Then we have $\mu^{N}(S)=1$.
\end{lem}

 %\noindent{\bf  Corollary  2.1}~{\it
 \begin{cor}(\cite{Pan13}, Corollary 2.3, p. 473)
 Let $\ell_1$ be a Lebesgue measure on $(0,1)$.
 Let $D$ be a set of all $\ell_1$-equidistributed sequences on
$(0,1)$. Then we have ~$\ell_1^{N}(D)=1$.
\end{cor}
\medskip

%\vspace{.08in} \noindent {\bf  Definition 2.3}

\begin{defn} Let $\mu$ be a
probability Borel measure on $R$ and $F$ be its corresponding
distribution function. A sequence $(x_k)_{k \in N}$ of elements of
$R$ is said to be $\mu$-equidistributed or $\mu$-uniformly
distributed on $R$ if for every interval $[a, b] (-\infty \le a <
b \le +\infty)$ we have
\begin{equation}
\lim_{n \to \infty} n^{-1} \# ([a,b] \cap \{x_1, \cdots, x_n \}
)=F(b)-F(a).\label{ccs}
\end{equation}
\end{defn}

%\noindent{\bf  Lemma 2.4}{\it ~
\begin{lem} (\cite{Pan13}, Lemma 2.4, p. 473)
 Let $(x_k)_{k \in N}$ be
$\ell_1$-equidistributed sequence on $(0,1)$, $F$ be a strictly
increasing  continuous distribution function on $R$ and $p$ be a
Borel probability measure on $R$ defined by $F$. Then
$(F^{-1}(x_k))_{k \in N}$ is $p$-equidistributed on $R$.
\end{lem}

%\noindent{\bf  Corollary  2.2}~{\it
\begin{cor}(\cite{Pan13}, Corollary 2.4, p. 473)
Let $F$ be a strictly
increasing continuous distribution  function on $R$ and $p_{F}$ be a
Borel probability measure on $R$ defined by $F$. Then for a set
$D_{F} \subset R^{N}$ of all $p$-equidistributed sequences  on $R$
we have :

(i) $D_{F}= \{ (F^{-1}(x_k))_{k \in N} : (x_k)_{k \in N} \in D
\}$;

(ii)~$p_{F}^N(D_{F})=1$.
\end{cor}

\begin{lem}
 Let $F_1$  and $F_2$ be two different strictly
increasing continuous distribution  functions on $R$ and $p_1$ and $p_2$ be
Borel probability measures on $R$ defined by $F_1$ and $F_2$, respectively.  Then there does not exist
a sequence of real numbers $(x_k)_{k \in \mathbf{N}}$ which simultaneously is $p_1$-equidistributed and $p_2$-equidistributed.
\end{lem}
\begin{proof}
Assume the contrary and let $(x_k)_{k \in \mathbf{N}}$ be such a sequence. Since
$F_1$  and $F_2$  are different there is a point $x_0 \in {R}$ such that $F_1(x_0)\neq F_2(x_0)$. The latter relation is not possible under our assumption, because
$(x_k)_{k \in \mathbf{N}}$ simultaneously is $p_1$-equidistributed and $p_2$-equidistributed, which implies
\begin{equation}
F_1(x_0)=\lim_{n \to \infty}n^{-1}\# ((-\infty,x_0] \cap \{x_1, \cdots,
x_n \} )=F_2(x_0).\label{ccs}
\end{equation}
\end{proof}

\begin{thm}
 Let $F_1$  and $F_2$ be two different strictly
increasing continuous distribution  functions on $R$ and $p_1$ and $p_2$ be
Borel probability measures on $R$ defined by $F_1$ and $F_2$, respectively.  Then the measures $p^N_1$ and $p^N_2$  are orthogonal.
\end{thm}
\begin{proof} Let $D_{F_1}$ and $D_{F_2}$  denote $p_1$-equidistributed and $p_2$-equidistributed sequences on $R$, respectively.  By Lemma 2.17  we know that
$D_{F_1} \cap D_{F_2}=\emptyset.$  By Corollary 2.16 we know that $p^N_1(D_{F_1})=1$ and $p^N_2(D_{F_2})=1$. This ends the proof of the theorem.
\end{proof}
\begin{defn} Let  $\{ \mu_i : i \in I\}$ be a  family  of probability measures defined on a measure space $(X,M)$. Let $S(X)$ be defined by
\begin{equation}
S(X)=\cap_{i \in I}\mbox{dom}(\overline{\mu}_i), \label{ccs}
\end{equation}
where $\overline{\mu}_i$ denotes a usual completion of the measure $\mu_i$ and $\mbox{dom}(\overline{\mu}_i)$ denotes the $sigma$-algebra of all $\overline{\mu}_i$-measurable subsets of $X$  for each $i \in I$. We say that the family $\{ \mu_i : i \in I\}$ is strong separable if  there  exists a partition $\{C_i :i  \in I\}$ of the space $X$ into elements of the $\sigma$-algebra $S(X)$ such that $\overline{\mu}_i(C_i)=1$ for each $i \in I$.
\end{defn}
\begin{defn} Let  $\{ \mu_i : i \in I\}$ be a  family  of probability measures defined on a measure space $(X,M)$. Let $L(I)$ denotes a minimal $\sigma$-algebra generated by all singletons of $I$  and  $S(X)$ be the $\sigma$-algebra of subsets of  $X$  defined by (2.7).
We say that a $(S(X),L(I))$-measurable mapping $T:X \to I$ is a consistent(or well-founded) estimate of an unknown parameter $i ~(i \in I)$ for the  family $\{ \mu_i : i \in I\}$  if the following condition
\begin{equation}
(\forall i)(i \in I \rightarrow \mu_i(T^{-1}(\{i\})=1))\label{ccs}
\end{equation}
holds true.
\end{defn}
\begin{lem}(\cite{Pan13}, Lemma 2.5, p.
474) Let  $\{ \mu_i : i \in I\}$ be a  family  of probability measures defined on a measure space $(X,M)$.  The following sentences are equivalent:

(i)~ The family  of probability measures $\{ \mu_i : i \in I\}$ is strong separable;

(ii) There exists  a consistent estimate of an unknown  parameter $i ~(i \in I)$ for the  family $\{ \mu_i : i \in I\}$.

\end{lem}
Now let  $X_1,X_2, \cdots$  be an infinite sampling of independent, equally distributed real-valued random
variables with unknown distribution function $F$. Assume that we know only that $F$ belongs to the family of distribution functions $\{ F_{\theta}: \theta \in \Theta\}$, where
$\Theta$ is  a non-empty set.
Using these infinite  sampling we want to estimate
an unknown distribution function $F$.
Let $\mu_{\theta}$ denotes
a Borel probability measure on the real axis ${\bf  R}$ generated by $F_{\theta}$ for $\theta \in \Theta$. We denote by $\mu_{\theta}^N$ an infinite power of the measure $\mu_{\theta}$, i.e.,
$\mu_{\theta}^N=\mu_{\theta} \times \mu_{\theta} \times \cdots.$

The triplet $({\bf  R}^N, \mathcal{B}({\bf  R}^N),\mu_{\theta}^N)_{\theta \in \Theta}$ is called a statistical structure described our infinite experiment.
%\vspace{.08in} \noindent {\bf  Definition 2.4}
\begin{defn}
 A Borel
measurable function  $T_n : {\bf  R}^n \to {\bf  R} ~(n \in N)$ is called a
consistent estimator of a parameter $\theta$ (in the sense of
everywhere convergence) for the family $(\mu_{\theta}^N)_{\theta
\in \Theta}$ if the following condition
\begin{equation}
\mu_{\theta}^N (\{ (x_k)_{k \in N} :~(x_k)_{k \in N} \in {\bf  R}^N~\&~
\lim_{n \to \infty}T_n(x_1, \cdots, x_n)=\theta \})=1\label{ccs}
\end{equation}
holds true for each $\theta \in \Theta$.
\end{defn}
%\vspace{.08in} \noindent {\bf  Definition 2.5}
\begin{defn}

 A Borel
measurable function  $T_n : {\bf  R}^n \to {\bf  R} ~(n \in N)$ is called a
consistent estimator of a parameter $\theta$ (in the sense of
convergence in probability) for the family
$(\mu_{\theta}^N)_{\theta \in \Theta}$   if for every $\epsilon>0$ and
$\theta \in \Theta$ the following condition
\begin{equation}
\lim_{n \to \infty} \mu_{\theta}^N (\{ (x_k)_{k \in N} :~(x_k)_{k
\in N} \in {\bf  R}^N~\&~ | T_n(x_1, \cdots, x_n)-\theta|>\epsilon \})=0\label{ccs}
\end{equation}
holds true.
\end{defn}

%\vspace{.08in} \noindent {\bf  Definition 2.6}

\begin{defn}
A Borel
measurable function  $T_n : {\bf  R}^n \to {\bf  R} ~(n \in N)$ is called a
consistent estimator of a parameter $\theta$ (in the sense of
convergence in distribution ) for the family
$(\mu_{\theta}^N)_{\theta \in \Theta}$ if for every continuous bounded
real valued function $f$ on ${\bf  R}$ the following condition
\begin{equation}
\lim_{n \to \infty} \int_{{\bf  R}^N}f(T_n(x_1, \cdots, x_n))d
\mu_{\theta}^N((x_k)_{k \in N})=f(\theta)\label{ccs}
\end{equation}
holds.
\end{defn}

%\vspace{.08in} \noindent {\bf  Remark 2.1}

\begin{rem}
  Following
\cite{Shiryaev80} (see, Theorem 2, p. 272), for the family
$(\mu_{\theta}^N)_{\theta \in R}$ we have:

(a) an existence of a consistent estimator of a parameter $\theta$
in the sense of everywhere convergence implies an existence of a
consistent estimator of a parameter $\theta$ in the sense of
convergence in probability;

(b) an existence of a consistent estimator of a parameter $\theta$
in the sense of convergence in probability implies an existence of
a consistent estimator of a parameter $\theta$ in the sense of
convergence in distribution.
\end{rem}

Now let  $L(\Theta)$  be  a minimal $\sigma$-algebra of subsets generated by all singletons of the set $\Theta$.

\begin{defn} A $(\mathcal{B}({\bf  R}^N), L(\Theta))$-measurable function  $T : {\bf  R}^N \to \Theta$ is called a infinite
sample consistent estimate (or estimator) of a parameter $\theta$ for the
family $(\mu_{\theta}^{N})_{\theta \in \Theta}$ if the
condition
\begin{equation}
{\mu_{\theta}^{N}} (\{ (x_k)_{k \in N}
:~(x_k)_{k \in N} \in {\bf  R}^N~\&~
 T( (x_k)_{k \in N} )=\theta \})=1 \label{ccs}
\end{equation}
holds true for each $\theta \in \Theta$.
\end{defn}
\begin{defn} An infinite
sample consistent estimate $T : {\bf  R}^N \to \Theta$ of a parameter $\theta$ for the
family $(\mu_{\theta}^{N})_{\theta \in \Theta}$ is called objective if $T^{-1}(\theta)$ is a Haar ambivalent set for each $\theta \in \Theta$. Otherwise, $T$ is called subjective.
\end{defn}
\begin{defn} An objective infinite
sample consistent estimate $T : {\bf  R}^N \to \Theta$ of a parameter $\theta$ for the
family $(\mu_{\theta}^{N})_{\theta \in \Theta}$ is called strong if each  $\theta_1,\theta_2 \in \Theta$ there exists ~an~isometric
~(with~respect~to~Tychonoff metric) transformation $A_{(\theta_1,\theta_2)}$ of ${\bf R}^N$ such that $A_{(\theta_1,\theta_2)}(T^{-1}( \theta_1))\Delta T^{-1}( \theta_2)$ is shy.
\end{defn}
\begin{defn} Following \cite{IbramSkor80}, the family $(\mu_{\theta}^N)_{\theta \in \Theta}$ is
called  strictly separated if there exists a family $(Z_{\theta})_{\theta \in \Theta}$ of Borel subsets of ${\bf  R}^N$  such  that

(i)~$\mu_{\theta}^N(Z_{\theta})=1$ for $\theta \in \Theta$;

(ii)~$Z_{\theta_1} \cap Z_{\theta_2}=\emptyset$ for all different
parameters $\theta_1$ and $\theta_2$ from $\Theta$.

(iii)~$\cup_{\theta \in \Theta}Z_{\theta}={\bf  R}^N.$
\end{defn}
%\vspace{.08in} \noindent {\bf  Remark 2.2}
\begin{rem}
 Notice  that an
existence of an infinite sample consistent estimator of a parameter $\theta$
for the family $(\mu_{\theta}^N)_{\theta \in \Theta}$ implies that the
family $(\mu_{\theta}^N)_{\theta \in \Theta}$ is strictly separated.
Indeed, if we  set $Z_{\theta}=\{ (x_k)_{k \in N} : (x_k)_{k \in
N} \in {\bf  R}^N ~\&~T((x_k)_{k \in N})=\theta\}$ for $\theta \in \Theta$,
then all conditions participated in the Definition 2.29 will be satisfied.
\end{rem}

\section{An objective infinite sample consistent estimate of an unknown  distribution function}

%\vspace{.08in} \noindent {\bf  Theorem 3.1}~{\it

\begin{thm}
Let  $\mathcal{F}$ be a
family of distribution functions on ${\bf  R}$ satisfying
the following properties:

(i)  each element of  $\mathcal{F}$ is strictly increasing and continuous;

(ii) there exists a point $x_{*}$ such that $F_1(x_{*})\neq F_2(x_{*})$ for each different $F_1, F_2 \in {\mathcal{F}}$.

Setting $\Theta=\{ \theta=F(x_*): F \in {\mathcal{F}}\}$
and  $F_{\theta}=F$ for $\theta=F(x_*)$, we get the following parametrization  ${\mathcal{F}}=\{ F_{\theta}:\theta \in \Theta\}.$
We denote by $\mu_{\theta}$ a Borel probability measure in ${\bf R}$ defined by $F_{\theta}$ for $\theta \in \Theta$.
Then a function $T_n :{\bf  R}^n \to
{\bf  R}$, defined by
\begin{equation}
T_n(x_1, \cdots, x_n)=\frac{\#( \{ x_1, \cdots, x_n \}
\cap (-\infty;x_{*} ])}{n} \label{ccs}
\end{equation}
for $(x_1, \cdots, x_n) \in {\bf  R}^n ~(n \in N)$,  is a consistent
estimator of a parameter $\theta$ for the family
$(\mu_{\theta}^N)_{\theta \in \Theta}$ in the sense of almost
everywhere convergence.
\end{thm}

%\noindent{\bf  Proof.}
\begin{proof}
It is clear that $T_n$ is Borel measurable
function for $n \in N$. For $\theta \in {\bf  R}$, we  set
\begin{equation}
A_{\theta}=\{ (x_k)_{k \in N}~:~ (x_k)_{k \in N}~ \mbox{is
~}\mu_{\theta}-{uniformly~ distributed ~on}~ {\bf  R} \}.\label{ccs}
\end{equation}
Following Corollary 2.16, we have $\mu_{\theta}^{N}(A_{\theta})=1$
for $\theta \in \Theta$.

For $\theta \in \Theta$, we get

\begin{equation*}
\mu_{\theta}^N (\{ (x_k)_{k \in N} \in
{\bf  R}^N~:~ ~\lim_{n \to \infty}T_n(x_1, \cdots, x_n)=\theta \})=\mu_{\theta}^N (\{ (x_k)_{k \in N} \in {\bf  R}^N
~:~
\end{equation*}
\begin{equation}
\lim_{n \to \infty} n^{-1}\#(\{ x_1, \cdots, x_n\} \cap
(-\infty;x_{*}])=F_{\theta}(x_{*}) \})
\ge \mu_{\theta}^N
(A_{\theta})=1.\label{ccs}
\end{equation}
\end{proof}

The following corollaries are simple consequences of Theorem
3.1 and Remark 2.25.

%\vspace{.08in} \noindent {\bf  Corollary 3.1}~{\it

\begin{cor}
 An estimator
$T_n$ defined by (3.1)  is a consistent estimator of a parameter
$\theta$ for the family $(\mu_{\theta}^N)_{\theta \in \Theta}$ in the
sense of convergence in probability.
\end{cor}

%\vspace{.08in} \noindent {\bf  Corollary 3.2}~{\it
\begin{cor}
 An estimator
$T_n$ defined by (3.1) is a consistent estimator of a parameter
$\theta$ for the family $(\mu_{\theta}^N)_{\theta \in \Theta}$ in the
sense of convergence in distribution.
\end{cor}
%\vspace{.08in} \noindent {\bf  Theorem 3.2}~{\it

\begin{thm} Let ${\mathcal{F}}=\{ F_{\theta}:\theta \in \Theta\}$  and  $(\mu_{\theta}^N)_{\theta \in \Theta}$ come from Theorem 3.1.  Let fix $\theta_0 \in \Theta$ and define an estimate $T^{(1)}_{\theta_0}: {\bf R}^N \to \Theta$ as follows:
$T^{(1)}_{\theta_0}((x_k)_{k \in N})=\overline{\lim}\widetilde{T_n}((x_k)_{k \in N})$ if
$\overline{\lim}\widetilde{T_n}((x_k)_{k \in N}) \in \Theta \setminus \{\theta_0\}$
and $T^{(1)}_{\theta_0}((x_k)_{k \in N})=\theta_0,$
otherwise, where $\overline{\lim}\widetilde{T_n} = \inf_n \sup_{m \ge
n}\widetilde{T_m}$
and
\begin{equation}
\widetilde{T_n}((x_k)_{k \in N})=n^{-1}\#(\{ x_1, \cdots,
x_n\} \cap (-\infty;x_*])) \label{ccs}
\end{equation}
for  $(x_k)_{k \in N} \in {\bf  R}^N $. Then $T^{(1)}_{\theta_0}$  is an objective  infinite sample consistent
estimator of a parameter $\theta$ for the family
$(\mu_{\theta}^N)_{\theta \in \Theta}$.
\end{thm}
%\noindent{\bf  Proof.}
\begin{proof}
Following \cite{Shiryaev80}(see, p. 189),
the function $ \overline{\lim}\widetilde{T_n}$  is Borel measurable which implies that the function $ \overline{\lim}\widetilde{T_n}$ is  $(\mathcal{B}({\bf R}^N),L(\Theta))$-measurable.  Following
Corollary 2.16, we have $\mu_{\theta}^{N}(A_{\theta})=1$ for
$\theta \in \Theta$, where $A_{\theta}$ is defined by (3.2). Hence we get
\begin{equation*}
\mu_{\theta}^N (\{ (x_k)_{k \in N} \in {\bf R}^N
~:~T^{(1)}_{\theta_0}(x_k)_{k \in N}=\theta \}) \ge
\end{equation*}
\begin{equation*}
\mu_{\theta}^N (\{ (x_k)_{k \in N} \in {\bf R}^N
~:~\overline{\lim}\widetilde{T_n}(x_k)_{k \in N}=\theta \}) \ge
\end{equation*}
\begin{equation*}
\mu_{\theta}^N (\{(x_k)_{k \in N} \in {\bf  R}^N
~:~~\overline{\lim} \widetilde{T_n}(x_k)_{k \in N}=
\end{equation*}
\begin{equation}
\underline{\lim}\widetilde{T_n}(x_k)_{k \in N}=F_{\theta}(x_{*}) \}) \ge
\mu_{\theta}^N(A_{\theta})=1\label{ccs}
\end{equation}
for $\theta
\in \Theta$.
Thus we have proved that the estimator
~${{\bf T}^{(1)}}_{\theta_0}$ is an infinite sample
consistent estimator of a parameter $\theta$ for the family
$(\mu_{\theta}^N)_{\theta \in \Theta}$.

Now let us show that ~${{\bf T}^{(1)}}_{\theta_0}$ is an objective infinite sample
consistent estimator of a parameter $\theta$ for the family
$(\mu_{\theta}^N)_{\theta \in \Theta}$.

Let us show that $B(\theta):=\big({{\bf T}^{(1)}}_{\theta_0}\big)^{-1}(\theta)$ is a Haar ambivalent set for each $\theta \in \Theta$.

Let $(x_k)_{k \in N}$ be $\mu_{\theta}$-uniformly distributed sequence on ${\bf  R}$. Then we get

\begin{equation}
\lim_{n \to \infty}n^{-1}\#(\{ x_1, \cdots,
x_n\} \cap (-\infty;x_*])=\theta.\label{ccs}
\end{equation}
Let consider a set

\begin{equation}
C(\theta)=\{ (y_k)_{k \in N}: y_k \le  x_k ~\mbox{if}~x_k \le x_*~\&~ y_k>x_k ~\mbox{if}~x_k>x_*\}.\label{ccs}
\end{equation}
Setting $J=\{ k : x_k \le x_*\},$ we claim that $C(\theta)-(x_k)_{k \in N}=A_J$, where $A_J$ comes from Lemma 2.7. Since any translate of Haar ambivalent set is again  Haar ambivalent set, we claim that  $C(\theta)$ is  Haar ambivalent set. A set $B(\theta)$ which contains the Haar ambivalent set $C(\theta)$  is non-shy. Since $\theta \in \Theta$ was taken arbitrary we deduce that each $B_{\theta}$ is Haar ambivalent set. The latter relation means that  the estimator
~${{\bf T}^{(1)}}_{\theta_0}$ is an objective infinite sample
consistent estimator of a parameter $\theta$ for the family
$(\mu_{\theta}^N)_{\theta \in \Theta}$.
\end{proof}
%\vspace{.08in} \noindent {\bf  Theorem 3.3}~{\it
\begin{thm}  Let ${\mathcal{F}}=\{ F_{\theta}:\theta \in \Theta\}$  and  $(\mu_{\theta}^N)_{\theta \in \Theta}$ come from Theorem 3.1.
Let fix  $\theta_0 \in \Theta$  and define  an estimate ${{\bf T}^{(2)}}_{\theta_0}: {\bf  R}^N \to \Theta$ as follows:
${{\bf T}^{(2)}}_{\theta_0}((x_k)_{k \in N})=\underline{\lim}\widetilde{T_n}((x_k)_{k \in N})$ if  $\underline{\lim}\widetilde{T_n}((x_k)_{k \in N}) \in \Theta \setminus \{\theta_0\}$ and ${{\bf T}^{(2)}}_{\theta_0}((x_k)_{k \in N})=\theta_0$ otherwise, where
$\underline{\lim}\widetilde{T_n}=\sup_n\inf_{m \ge n}\widetilde{T_m}$ and
\begin{equation}
\widetilde{T_n}((x_k)_{k \in N})=n^{-1}\#(\{ x_1, \cdots,
x_n\} \cap (-\infty;x_*]) \label{ccs}
\end{equation}
for $(x_k)_{k \in N} \in {\bf  R}^N$. Then ${{\bf T}^{(2)}}_{\theta_0}$  is an objective infinite sample consistent
estimator of a parameter $\theta$ for the family
$(\mu_{\theta}^N)_{\theta \in \Theta}$.
\end{thm}
%\noindent{\bf  Proof.}
\begin{proof}
Following \cite{Shiryaev80}(see, p. 189),
the function $\underline{\lim}\widetilde{T_n}$  is Borel measurable which implies that the function $ \overline{\lim}\widetilde{T_n}$ is  $(\mathcal{B}({\bf  R}^n),L(\Theta))$-measurable. Following
Corollary 2.16, we have $\mu_{\theta}^{N}(A_{\theta})=1$ for
$\theta \in \Theta$, where $A_{\theta}$ is defined by (3.2). Hence we get
\begin{equation*}
\mu_{\theta}^N (\{(x_k)_{k \in N} \in {\bf  R}^N
~:~{{\bf T}^{(2)}}_{\theta_0}(x_k)_{k \in N}=\theta \}) \ge
\end{equation*}
\begin{equation*}
\mu_{\theta}^N (\{ (x_k)_{k \in N} \in {\bf  R}^N
~:~\underline{\lim}\widetilde{T_n}(x_k)_{k \in N}=\theta \}) \ge
\end{equation*}
\begin{equation*}
\mu_{\theta}^N (\{ (x_k)_{k \in N} \in {\bf  R}^N
~:~~\overline{\lim} \widetilde{T_n}(x_k)_{k \in N}=
\end{equation*}
\begin{equation}
\underline{\lim}\widetilde{T_n}(x_k)_{k \in N}=F_{\theta}(x_{*}) \}) \ge
\mu_{\theta}^N(A_{\theta})=1\label{ccs}
\end{equation}
for $\theta \in \Theta$.

Thus we have proved that the estimator
~${{\bf T}^{(2)}}_{\theta_0}$ is an infinite sample
consistent estimators of a parameter $\theta$ for the family
$(\mu_{\theta}^N)_{\theta \in \Theta}$.

Now let us show that ~${{\bf T}^{(2)}}_{\theta_0}$ is an objective infinite sample
consistent estimator of a parameter $\theta$ for the family
$(\mu_{\theta}^N)_{\theta \in \Theta}$.

Let us show that  $B(\theta)=\big({{\bf T}^{(2)}}_{\theta_0}\big)^{-1}(\theta)$ is Haar ambivalent set for each $\theta \in \Theta$.

Let $(x_k)_{k \in N}$ be $\mu_{\theta}$-uniformly distributed sequence. Then we get
\begin{equation}
\lim_{n \to \infty}n^{-1}\#(\{ x_1, \cdots,
x_n\} \cap (-\infty;x_*])=\theta.\label{ccs}
\end{equation}
Let consider a set
\begin{equation}
C(\theta)=\{ (y_k)_{k \in N}: (y_k)_{k \in N} \in {\bf R}^N~\&~y_k \le  x_k ~\mbox{if}~x_k \le x_*~\&~ y_k>x_k ~\mbox{if}~x_k>x_*\}.\label{ccs}
\end{equation}
Setting $J=\{ k :  x_k \le x_*\},$  we deduce that $C(\theta)-(x_k)_{k \in N}=A_J$, where $A_J$ comes from  Lemma 2.7. Since any translate of Haar ambivalent set  is again Haar ambivalent set, we claim that  $C(\theta)$ is Haar ambivalent set. A set $B(\theta)$ which contains the Haar ambivalent set $C(\theta)$  is non-shy. Since $\theta \in \Theta$ was taken arbitrary we deduce that each $B_{\theta}$ is Haar ambivalent set. The latter relation means that  the estimator
~${{\bf T}^{(2)}}_{\theta_0}$ is an objective infinite sample
consistent estimator of a parameter $\theta$ for the family
$(\mu_{\theta}^N)_{\theta \in \Theta}$.
\end{proof}
%\vspace{.08in} \noindent {\bf  Remark 3.2}~{\it
\begin{rem} It can be shown that  Theorems 3.4 and 3.5 extend the recent result obtained in \cite{Pan15}(see Theorem 3.1). Indeed, let consider  the linear
one-dimensional stochastic system
\begin{equation}
(\xi_k)_{k \in N}=(\theta_k)_{k \in
N}+(\Delta_k)_{k \in N}, \label{ccs}
\end{equation}
where $(\theta_k)_{k \in N} \in {\bf  R}^N$ is a sequence of
useful signals, $(\Delta_k)_{k \in N}$ is sequence of
independent identically distributed random variables (the
so-called generalized ``white noise" ) defined on some probability
space $(\Omega,{\mathcal{F}},P)$ and  $(\xi_k)_{k \in N}$
is a sequence of transformed signals. Let $\mu$ be a Borel
probability measure on ${\bf  R}$ defined by a random variable
$\Delta_1$. Then the $N$-power  of the measure $\mu$
denoted by $\mu^{N}$  coincides with  the  Borel
probability measure on ${\bf  R}^N$ defined by the generalized
``white noise", i.e.,
\begin{equation}
(\forall X)(X \in {\mathcal{B}}{({\bf  R}^N}) \rightarrow
{\mu^{N}}(X)=P(\{ \omega : \omega \in \Omega ~\&~
(\Delta_k(\omega))_{k \in N} \in X\})),\label{ccs}
\end{equation}
where $\mathcal{B}({\bf  R}^N)$ is the Borel  $\sigma$-algebra
of subsets of ${\bf  R}^N$.

  Following \cite{IbramSkor80},   a general decision in the information transmission theory is that the Borel probability measure
$\lambda$, defined  by the sequence of transformed signals
$(\xi_k)_{k \in N}$ coincides with $\big(\mu^{N}
\big)_{\theta_0}$ for some $\theta_0 \in \Theta$ provided that
\begin{equation}
(\exists \theta_0)( \theta_0 \in \Theta \rightarrow (\forall X)(X
\in {\mathcal{B}({\bf  R}^N)} \rightarrow
\lambda(X)=\big(\mu^{N} \big)_{\theta_0}(X))),
\label{ccs}
\end{equation}
where $~\big(\mu^{N}
\big)_{\theta_0}(X)=\mu^N(X-\theta_0)$ for $X \in
\mathcal{B}({\bf  R}^N)$.

In \cite{Pan13} has been considered  a particular case of the above model (3.12) for which
\begin{equation}
(\theta_k)_{k \in N}  \in \{ (\theta,\theta,\cdots) : \theta \in  {\bf R} \}.\label{ccs}
\end{equation}
For $\theta \in  {\bf R}$, a measure $\mu_{\theta}^{N}$
defined by
\begin{equation}
\mu_{\theta}^N=\mu_{\theta} \times  \mu_{\theta} \times \cdots,\label{ccs}
\end{equation}
where $\mu_{\theta}$ is a $\theta$-shift of $\mu$ (i.e.,
$\mu_{\theta}(X)=\mu(X-\theta)$ for $X \in
\mathcal{B}( {\bf R})$), is called the $N$-power of the
$\theta$-shift of $\mu$ on $ {\bf R}$.

Let denote by $F_{\theta}$ a distribution function defined by $\mu_{\theta}$ for $\theta \in \Theta$.
Notice that the family  ${\mathcal{F}}=\{ F_{\theta} : \theta  \in \Theta \}$  satisfies all conditions participated in Theorems 3.1. Indeed, under  $x_*$ we can take the zero of the real axis. Then following  Theorems 3.4 and 3.5, estimators $T^{(1)}_{\theta_0}$ and $T^{(2)}_{\theta_0}$  are objective infinite sample consistent
estimators of a useful signal $\theta$ in the linear one-dimensional stochastic system (3.12). Notice that  these estimators exactly coincide with estimators constructed in \cite{Pan15}(see  Theorem 3.1).

\end{rem}
\begin{thm} Let  $\mathcal{F}$ be a  family  of all strictly increasing and continuous distribution functions in ${\bf  R}$  and $p_{F}$ be a Borel probability measure
on ${\bf R}$ defined by $F$ for each $F \in \mathcal{F}$. Then the family of Borel probability measures  $\{ p_F^N : F \in \mathcal{F})\} $ is strong separable.
\end{thm}
\begin{proof} We denote by $D_{F}$ the set of all  $p_F$-equidistributed  sequences on ${\bf R}$  for each $F \in \mathcal{F}$.  By Lemma 2.17  we know that
$D_{F_1} \cap D_{F_2}=\emptyset$ for each different $F_1,F_2 \in \mathcal{F}$.  By Corollary 2.16 we know that $p^N_F(D_{F})=1$    for each $F \in \mathcal{F}$.  Let fix  $F_0 \in \mathcal{F}$ and define a family $(C_F)_{F  \in \mathcal{F}}$ of subsets of ${\bf  R}^N$ as follows: $C_F=D_F$ for $F \in \mathcal{F} \setminus \{F_0\}$ and $C_{F_0}={\bf R}^N \setminus \cup_{ F \in \mathcal{F} \setminus \{F_0\}}D_F$. Notice that since $D_F$ is a Borel subset of ${\bf R}^N$ for each $F \in \mathcal{F}$ we claim that $C_F \in S({\bf R}^N)$ for each $F \in \mathcal{F} \setminus \{F_0\}$, where $S({\bf R}^N)$ comes from Definition 2.19. Since
$\overline{p_F^N}({\bf R}^N \setminus \cup_{ F \in \mathcal{F}}D_F)=0$ for each  $F \in \mathcal{F}$, we deduce that ${\bf R}^N \setminus \cup_{ F \in \mathcal{F}}D_F \in \cap_{F \in \mathcal{F}}\mbox{dom}(\overline{p_F^N})=S({\bf R}^N)$. Since $S({\bf  R}^N)$ is the $\sigma$-algebra, we claim that $C_{F_0} \in S({\bf  R}^N)$, because
$\overline{p_F^N}({\bf R}^N \setminus \cup_{ F \in \mathcal{F}}D_F)=0$ for each  $F \in \mathcal{F}$(equivalently, ${\bf  R}^N \setminus \cup_{ F \in \mathcal{F}}D_F  \in S({\bf  R}^N)$), and
\begin{equation}
C_{F_0}={\bf  R}^N \setminus \cup_{ F \in \mathcal{F} \setminus \{F_0\}}D_F=({\bf  R}^N \setminus \cup_{ F \in \mathcal{F}}D_F)\cup D_{F_0}.\label{ccs}
\end{equation}
This ends the proof of the theorem.
\end{proof}

\begin{rem} By virtue the results of  Lemma 2.21 and  Theorem 3.7 we get that  there exists  a consistent estimate of an unknown distribution function  $F  ~(F \in \mathcal{F})$ for the family of Borel probability measures  $\{ p_F^N : F \in \mathcal{F}\}$, where $\mathcal{F}$ comes from Theorem 3.7. This estimate $T : {\bf R}^N  \to \mathcal{F}$ is defined by :  $T((x_k)_{k \in N})=F$ if $(x_k)_{k \in N} \in C_F$ , where the family $(C_F)_{F  \in \mathcal{F}}$ of subsets of ${\bf  R}^N$ also comes from Theorem 3.7. Notice that this result  extends the main result established  in  \cite{Pan13}(see  Lemma 2.6, p. 476).
\end{rem}
At end of this section we state the following
\begin{prob} Let  $\mathcal{F}$ be a  family  of all strictly increasing and continuous distribution functions on ${\bf  R}$  and $p_{F}$ be a Borel probability measure
in ${\bf R}$ defined by $F$ for each $F \in \mathcal{F}$. Does there exist an objective infinite sample consistent estimate of an  unknown distribution function $F$ for the  family of Borel probability measures  $\{ p_F^N : F \in \mathcal{F})\} $?
\end{prob}

\section{An effective construction of the strong objective infinite sample consistent estimate of a useful signal in the linear one-dimensional stochastic model}

In \cite{Pan14},  the examples  of {\it objective} and {\it strong objective}
infinite sample consistent estimates(\cite{Pan14}, $T^{\star}$(p. 63), $T^{\circ}$(p. 67))
of a useful signal
in the linear one-dimensional stochastic model were constructed by using the axiom of choice
and a certain partition of the non-locally compact abelian Polish
group ${R^N}$ constructed in \cite{Pan09}.

In this section,  in the same model  we present an effective example of the  strong objective infinite sample consistent estimate of a useful signal constructed  in \cite{Pan14-1}.

For each real number $a \in {R}$, we denote by $\{a\}$ its fractal part in the decimal  system.

\begin{thm}Let consider the linear one-dimensional stochastic model (3.12), for which "white noise" has a infinite absolute moment of the first
order and its moment of the first order is equal to zero.  Suppose that the Borel probability measure
$\lambda$, defined  by the sequence of transformed signals
$(\xi_k)_{k \in N}$  coincides with $\big(\mu_{\theta_0}^{N}
\big)$ for some $\theta_0 \in [0,1]$.
Let $T : {R^N} \to [0,1]$ be defined by:
$T((x_k)_{k \in N})=\{\lim_{n \to \infty}\frac{\sum_{k=1}^nx_k}{n}\}$ if
$\lim_{n \to \infty}\frac{\sum_{k=1}^nx_k}{n}\neq 1$, $T((x_k)_{k \in N})=1$ if
$\lim_{n \to \infty}\frac{\sum_{k=1}^nx_k}{n}=1$, and
$T((x_k)_{k \in N})=\sum_{k \in N}\frac{\chi_{(0,+\infty)}(x_k)}{2^k}$, otherwise, where
$\chi_{(0,+\infty)}(\cdot)$ denotes an indicator function of the set $(0,+\infty)$ defined on the real axis ${R}$.
Then  $T$ is a strong objective infinite sample consistent estimate of the parameter $\theta$
for the statistical structure $({R}^{N},\mathcal{B}({R}^{N}),\mu_{\theta}^{N})_{\theta
\in \Theta}$ describing the
linear one-dimensional stochastic system (3.12).
\end{thm}
\begin{proof} {\bf Step 1.} We have to show that $T$ is an infinite sample consistent estimate of the parameter $\theta$
for the statistical structure $({R}^{N},\mathcal{B}({R}^{N}),\mu_{\theta}^{N})_{\theta
\in \Theta}$
 and
$T^{-1}(\theta)$ is a~Haar ~ambivalent set for each $\theta =\sum_{k=1}^{\infty}\frac{\theta_k}{2^k} \in \Theta$, where
$\sum_{k=1}^{\infty}\frac{\theta_k}{2^k}$ is representation of the number $\theta$  in the binary system.

Indeed, we have

\begin{equation}
(\forall \theta)(\theta \in (0,1) \rightarrow T^{-1}(\theta)=
 (B_{H(\theta)} \setminus S) \cup \cup_{z \in {Z}} S_{\theta +z}),\label{ccs}
\end{equation}
where $H(\theta)=\{ k : k \in N~\&~\theta_k=1\}$, $B_{H(\theta)}= \theta - A_{H(\theta)}$,  $A_{H(\theta)}$ comes from Lemma 2.7,
\begin{equation}
S= \{ (x_k)_{k \in N} \in {\bf R}^N : \hbox{exists a finite limit} ~\lim_{n \to \infty} \frac{\sum_{k=1}^nx_k}{n} \} \label{ccs}
\end{equation}
and
\begin{equation}
S_{\theta +z}= \{ (x_k)_{k \in N} \in {\bf R}^N : \lim_{n \to \infty} \frac{\sum_{k=1}^nx_k}{n}=\theta +z \} \label{ccs}
\end{equation}
for each $\theta \in \Theta$ and $z \in Z$.
Notice that the set $S$ like $\cup_{z \in {Z}} S_{\theta +z}$ is Borel shy set (see \cite{Pan14}, Lemma 4.14, p. 60). Taking into account this fact,  the results of Lemmas 2.4 and  2.7,  invariance of Haar ambivalent sets under translations and symmetric transformation  and  the simple statement that difference  of non shy  and shy sets is non shy, we deduce that $T^{-1}(\theta)$ is a Borel measurable Haar ambivalent sets
 for each $\theta \in \Theta$.

 Notice that
 \begin{equation}
T^{-1}(1)=(B_{H(1)} \setminus S)\cup S_1=(B_{N} \setminus S)\cup S_1\label{ccs}
\end{equation}
 and
 \begin{equation}
T^{-1}(0)=
 (B_{H(0)} \setminus S) \cup \cup_{z \in {Z}\setminus \{1\}} S_{0 +z}=
 (B_{\emptyset} \setminus S) \cup \cup_{z \in {Z}\setminus \{1\}} S_{0 +z},\label{ccs}
\end{equation}
which also are Borel measurable Haar ambivalent sets.

 Now it is not hard to show that $T$ is $(\mathcal{B}({R^N}), L(\Theta))$- measurable because
 the class $\mathcal{B}({R^N})$ is closed under countable family of set-theoretical operations and each element
 of $L(\Theta)$ is  countable or co-countable in the interval $\Theta=[0,1]$.
  Since $S_{\theta} \subseteq   T^{-1}(\theta)$ for $\theta \in \Theta$, we deduce that $\mu_{\theta}(T^{-1}(\theta))=1$.
The later relation means that $T$ is an infinite sample consistent estimate of a parameter $\theta$.

{\bf Step 2.} Let us show that for each different ~$ \theta_1,\theta_2 \in [0,1]$
there exists ~an~isometric
~(with~respect~to~Tychonoff~metric)~transformation~$A_{(\theta_1,\theta_2)}$ such that
\begin{equation}
A_{(\theta_1,\theta_2)}(T^{-1}( \theta_1))\Delta T^{-1}( \theta_2)\label{ccs}
\end{equation}
is shy.

We define $A_{(\theta_1,\theta_2)}$ as follows: for $(x_k)_{k \in N}\in {R^{N}}$
  we put $A_{(\theta_1,\theta_2)}((x_k)_{k \in N})=
  (y_k)_{k \in N}$, where $y_k=-x_k$ if $k \in H(\theta_1) \Delta H(\theta_2)(:=(H(\theta_1) \setminus H(\theta_2))\cup
  ( H(\theta_2) \setminus H(\theta_1))$ and $y_k=x_k$ otherwise.
It is obvious that $A_{(\theta_1,\theta_2)}$ is isometric (with respect to Tychonoff metric) transformation of the ${R^N}$.

Notice that
\begin{equation}
A_{(\theta_1,\theta_2)}(T^{-1}( \theta_1))\Delta T^{-1}( \theta_2) \subseteq \cup_{k \in N}
\{0\}_k \times {R}^{N \setminus \{k\}} \cup S.\label{ccs}
\end{equation}
Since both sets  $\cup_{k \in N}\{0\}_k \times {R}^{N \setminus \{k\}}$ and $S$ are shy, by  Lemmas 2.4
and Definition 2.2 we claim that the set
\begin{equation}
A_{(\theta_1,\theta_2)}(T^{-1}( \theta_1))\Delta T^{-1}( \theta_2)\label{ccs}
\end{equation}
is also  shy.

This ends the proof of the theorem.
\end{proof}

\section{On infinite sample consistent estimates of an unknown  probability density function}

Let  $X_1,X_2, \cdots$  be independent identically  distributed real-valued random
variables having a common probability density function $f$.
After a so-called kernel class of estimates $f_n$ of $f$ based on $X_1,X_2, \cdots, X_n$ was introduced by Rosenblatt \cite{Rosenblatt56},various convergence properties of these estimates have been studied. The stronger result in this direction was due to Nadaraya \cite{Nadaraya1965}  who proved that if $f$ is uniformly continuous then for a large class of kernels the estimates $f_n$  converges uniformly on the real line to $f$ with probability one. In \cite{Schuster1969}, has been shown that the above assumptions on $f$  are necessary for this type of convergence. That is, if $f_n$ converges uniformly to a function $g$ with probability one, then $g$ must be uniformly continuous and the distribution $F$ from which we are sampling must be absolutely continuous with $F^{'}(x)=g(x)$ everywhere. When in addition to the mentioned above, it is assumed that $f$ and its first $r+1$ derivatives are bounded, it is possible to show that  how to construct estimates $f_n$ such that $f_n^{(s)}$ converges uniformly to  $f^{(s)}$ as a given rate with probability one for $s=0,\cdots, r.$
Let $f_n(x)$ be a kernel estimate based on  $X_1,X_2, \cdots, X_n$  from $F$ as given in \cite{Rosenblatt56}, that is
\begin{equation}
f_n(x)=(na_n)^{-1}\sum_{i=1}^nk(\frac{X_i-x}{a_n})\label{ccs}
\end{equation}
where $(a_n)_{n \in N}$ is a sequence of positive numbers converging to zero and $k$ is a probability density function such that $\int_{-\infty}^+\infty|x|k(x)dx$ is finite and $k^{(s)}$ is continuous function of bounded variation for $s=0,\cdots,r$. The density function of the standard normal, for example, satisfies all these conditions.

In the sequel we need the following wonderful statement.

\begin{lem}( \cite{Schuster1969}, Theorem 3.11, p. 1194)  A necessary and sufficient condition for
\begin{equation}
\lim_{n \to \infty}\sup_{x \in {\bf R}}|f_n(x) - g(x)|=0 \label{ccs}
\end{equation}
with probability one for a function $g$ is that $g$ be the uniformly continuous derivative of $F$.
\end{lem}

Let  $X_1,X_2, \cdots$  be independent and identically  distributed real-valued random
variables with an unknown probability density function $f$. Assume that we know that $f$ belongs to the class of probability density function $\mathcal{SC}$ each element of which is uniformly continuous.

Let denote by $\ell^{\infty}(R)$ an infinite-dimensional  non-separable Banach space  of all bounded real-valued functions on ${\bf R}$  equipped with norm $||\cdot||_{\infty}$ defined by
\begin{equation}
||h||_{\infty}=\sup_{x \in {\bf R}}|h(x)| \label{ccs}
\end{equation}
for all $h \in \ell^{\infty}(R)$. We say that $(\ell^{\infty}(R))\lim_{n \to \infty}h_n=h_0$ if $\lim_{n \to \infty}||h_n-h_0||_{\infty}=0.$

\begin{thm}Let $\phi$ denotes a normal density function.
We set $\Theta=\mathcal{SC}$. Let $\mu_{\theta}$ be a Borel probability measure on ${\bf R}$ with probability density function $\theta \in \Theta$.
Let  fix $\theta_0 \in \Theta$. For each $(x_i)_{i \in N}$ we set: $T_{\mathcal{SC}}((x_i)_{i \in N})=(\ell^{\infty}(R))\lim_{n \to \infty} f_n$ if this limit exists and is in $\Theta \setminus \{\theta_0\}$, and
 $T_{\mathcal{SC}}((x_i)_{i \in N})=\theta_0$, otherwise. Then $T_{\mathcal{SC}}$ is a consistent infinite-sample estimate of an unknown parameter  $\theta$ for the family $(\mu^N_{\theta})_{\theta \in \Theta}$.
 \end{thm}
 \begin{proof} By Lemma 5.1, for each $\theta \in \Theta$ we have

\begin{equation*}
\mu_{\theta}^N(\{ (x_i)_{i \in N} \in {\bf R}^N : T_{\mathcal{SC}}((x_i)_{i \in N})=\theta \})\ge
\end{equation*}

\begin{equation*}
\mu_{\theta}^N(\{ (x_i)_{i \in N} \in {\bf R}^N : (\ell^{\infty}(R))\lim_{n \to \infty}f_n= \theta \})=
\end{equation*}

\begin{equation*}
\mu_{\theta}^N(\{ (x_i)_{i \in N} \in {\bf R}^N : \lim_{n \to \infty} ||f_n - \theta||_{\infty}=0 \})= 1. \label{ccs}
\end{equation*}

This ends the proof of theorem.

\end{proof}

Concerning with Theorem 5.2 we state the following problems.

\begin{prob} Let $T_{\mathcal{SC}}$ comes from the Theorem 5.2. Is $T_{\mathcal{SC}}$  an objective infinite sample consistent estimate of the parameter $\theta$ for the  family $(\mu^N_{\theta})_{\theta \in \Theta}$?
\end{prob}

\begin{prob} Let the statistical structure $\{ ({\bf R}^N, \mathcal{B}({\bf R}^N),  \mu^N_{\theta}) : \theta \in \Theta\}$ comes from the Theorem 5.2. Does there exist  an objective (or strong objective) infinite sample consistent estimate of the parameter $\theta$ for the  family $(\mu^N_{\theta})_{\theta \in \Theta}$ ?
\end{prob}

Let  $X_1,X_2, \cdots$  be independent and identically  distributed real-valued random
variables with positive continuous probability density function $f$. Assume that we know that $f$ belongs to the separated class $\mathcal{A}$ of positive continuous probability densities  provided that there is a point $x_*$ such that $g_1(x_*) \neq g_2(x_*)$ for each $g_1,g_2 \in  \mathcal{A}$.  Suppose that we have an infinite sample $(x_k)_{k \in N}$ and we want to estimate an unknown probability density function. Setting $\Theta=\{ \theta= g(x_*): g \in \mathcal{A}\}$, we can parameterise the family $\mathcal{A}$ as follows:$\mathcal{A}=\{ f_{\theta}: \theta \in \Theta\}$, where $f_{\theta}$ is such a unique element $f$ from the family $\mathcal{A}$ for which $f(x_*)=\theta$. Let $\mu_{\theta}$ be a Borel probability measure defined by the probability density function $f_{\theta}$  for each $\theta \in \Theta$. It is obvious that $\{ ({\bf  R}^N, \mathcal{B}({\bf  R}^N), \mu^N_{\theta}): \theta \in \Theta\}$ will be the statistical structure described our experiment.

\begin{thm} Let $(h_m)_{m \in N}$ be a sequence of a strictly decreasing sequence of positive numbers tending to zero. Let fix $\theta_0 \in \Theta$. For each $(x_k)_{k \in N} \in {\bf  R}^N$ we put
\begin{equation}
T((x_k)_{k \in N})=\lim_{m \to \infty}\lim_{n \to \infty} \frac{\#(\{ x_1,\cdots, x_n\} \cap [x_*-h_m, x_*+h_m])}{2 n h_m}\label{ccs}
\end{equation}
if this repeated limit exists and belongs to the set $\Theta \setminus \{\theta\}$, and  $T((x_k)_{k \in N})=\theta_0$, otherwise.  Then $T$ is an infinite sample consistent estimate of the parameter $\theta$ for the  family $(\mu^N_{\theta})_{\theta \in \Theta}$.
\end{thm}
\begin{proof} For each $\theta \in \Theta$, we put
\begin{equation}
A_{\theta}=\{ (x_k)_{k \in N}: (x_k)_{k \in N} \in {\bf  R}^N~\&~(x_k)_{k \in N}~\mbox{is}~\mu_{\theta}-{equidistributed}\}.\label{ccs}
\end{equation}
By Corollary 2.16 we know that $\mu^N_{\theta}(A_{\theta})=1$ for each $\theta \in \Theta$.

For each $\theta \in \Theta$, we have
\begin{equation*}
\mu^N_{\theta}(T^{-1}(\theta))= \mu^N_{\theta}(\{ (x_k)_{k \in N} \in {\bf  R}^N~:~
T((x_k)_{k \in N})=\theta\})\ge
\end{equation*}
\begin{equation*}
\mu^N_{\theta}(\{(x_k)_{k \in N} \in A_{\theta}~:~T((x_k)_{k \in N})=\theta\})=
\end{equation*}
\begin{equation*}
\mu^N_{\theta}(\{ (x_k)_{k \in N} \in A_{\theta}~:~\lim_{m \to \infty} \frac{F_{\theta}(x_*+h_m)-F_{\theta}(x_*-h_m)}{2h_m }=\theta \})=
\end{equation*}
\begin{equation*}
\mu^N_{\theta}(\{(x_k)_{k \in N} \in A_{\theta}~: \lim_{m \to \infty} \frac{\int_{x_*-h_m}^{x_*+h_m}f_{\theta}(x)dx}{2h_m }=\theta \})=
\end{equation*}
\begin{equation}
\mu^N_{\theta}(\{(x_k)_{k \in N} \in A_{\theta}~: f_{\theta}(x^*)=\theta \})=
\mu^N_{\theta}(A_{\theta})=1 \label{ccs}
\end{equation}
\end{proof}
Concerning with Theorem 5.3 we state the following
\begin{prob} Let $T$ comes from the Theorem 5.3. Is $T$ an objective infinite sample consistent estimate of the parameter $\theta$ for the  family $(\mu^N_{\theta})_{\theta \in \Theta}$?
\end{prob}
\begin{prob}  Let the statistical structure $\{ ({\bf R}^N, \mathcal{B}({\bf R}^N),  \mu^N_{\theta}) : \theta \in \Theta \}$ comes from the Theorem 5.3. Does there exist  an objective (or strong objective)  infinite sample consistent estimate of the parameter $\theta$ for the  family $(\mu^N_{\theta})_{\theta \in \Theta}$ ?
\end{prob}
\begin{ex} Let  $X_1,X_2, \cdots$  be independent normally  distributed real-valued random  variables with parameters $(a,\sigma)$ where  $a$ is a mean and $\sigma$ is a standard deviation. Suppose that we know the  mean $a$ and  want to estimate an unknown standard deviation $\sigma$ by an infinite sample $(x_k)_{k \in N}$.  For each $\sigma>0$, let denote by $\mu_{\sigma}$ the Gaussian  probability measure on ${\bf  R}$ with parameters $(a, \sigma)$(here $a \in {\bf  R} $ is fixed).  Let $(h_m)_{m \in N}$ be a sequence of a strictly decreasing sequence of positive numbers tending to zero.

 By virtue of Theorem 5.4 we know that  for each $\sigma>0$ the following condition
\begin{equation*}
\mu^N_{\sigma}(\{ (x_k)_{k \in N} \in {\bf  R}^N~\&~
\end{equation*}
\begin{equation}
 \lim_{m \to \infty}\lim_{n \to \infty} \frac{\#(\{ x_1,\cdots, x_n\} \cap [a-h_m, a+h_m])}{2nh_m}=\frac{1}{\sqrt{2\pi}\sigma} \})=1\label{ccs}
\end{equation}
holds true.

Let fix $\sigma_0 >0$.  For $(x_k)_{k \in N} \in {\bf  R}^N$ we put
\begin{equation}
T_1((x_k)_{k \in N})=\lim_{m \to \infty}\lim_{n \to \infty} \frac{2nh_m}{\sqrt{2\pi} \#(\{ x_1,\cdots, x_n\} \cap [a-h_m, a+h_m])}\label{ccs}
\end{equation}
if this limit exists and belongs to the set $(0,+\infty) \setminus \{\sigma_0\}$, and  $T_1((x_k)_{k \in N})=\sigma_0$, otherwise.   Then for each $\sigma>0$ we get
\begin{equation}
\mu^N_{\sigma}(\{ (x_k)_{k \in N}: (x_k)_{k \in N} \in {\bf  R}^N~\&~ T_1((x_k)_{k \in N})= \sigma \})=1\label{ccs}
\end{equation}
which means that $T_1$ is an infinite sample consistent estimate of the  standard deviation $\sigma$ for the  family $(\mu^N_{\sigma})_{\sigma >0}$.
\end{ex}

\begin{thm}Let  $X_1,X_2, \cdots$  be independent normally  distributed real-valued random  variables with parameters $(a,\sigma)$,  where  $a$ is a mean and $\sigma$ is a standard deviation. Suppose that we know the  mean $a$. Let $(a_n)_{n \in N}$ be  a sequence of positive numbers converging to zero and $\phi$ be a standard Gaussian density function in ${\bf R}$.
We denote by $\mu_{\sigma}$ a Borel Gaussian probability measure in ${\bf R}$ with parameters  $(a,\sigma)$ for each $ \sigma \in \Sigma=(0,\infty)$.
Let fix $\sigma_0 \in \Sigma$. Let define  an estimate $T^{(1)}_{\sigma_0}: {\bf R}^N \to \Sigma$ as follows :
$T^{(1)}_{\sigma_0}((x_k)_{k \in N})=\overline{\lim}\widetilde{T^{(1)}_n}((x_k)_{k \in N})$ if
$\overline{\lim}\widetilde{T^{(1)}_n}((x_k)_{k \in N}) \in \Sigma \setminus \{\sigma_0\}$ and
$
T^{(1)}_{\sigma_0}((x_k)_{k \in N})=\sigma_0,$
otherwise, where $\overline{\lim}\widetilde{T^{(1)}_n} := \inf_n \sup_{m \ge
n}\widetilde{T^{(1)}_m}$
and
\begin{equation}
\widetilde{T^{(1)}_n}((x_k)_{k \in N})=T^{(1)}_n(x_1,\cdots,x_n)=\frac{1}{\sqrt{2 \pi} (na_n)^{-1}\sum_{i=1}^n\phi(\frac{x_i-a}{a_n})} \label{ccs}
\end{equation}
for  $(x_k)_{k \in N} \in {\bf  R}^N $. Then  $T^{(1)}_{\sigma_0}$  is an infinite sample consistent
estimator of a parameter $\sigma$ for the family
$(\mu_{\sigma}^N)_{\sigma \in \Sigma}$.
\end{thm}
%\noindent{\bf  Proof.}
\begin{proof}
Following \cite{Shiryaev80}(see, p. 189),
the function $ \overline{\lim}\widetilde{T^{(1)}_n}$  is Borel measurable which implies that the function $ \overline{\lim}\widetilde{T^{(1)}_n}$ is  $(\mathcal{B}({\bf R}^N),L(\Sigma))$-measurable.

For each $\sigma \in \Sigma$ we put
\begin{equation}
A_{\sigma}=\{ (x_k)_{k \in N} \in {\bf R}^N: lim_{n \to \infty}(na_n)^{-1}\sum_{i=1}^n\phi(\frac{x_i-a}{a_n})=f_{\sigma}(a)\}.\label{ccs}
\end{equation}
 Since uniformly convergence implies pointwise convergence, by  Lemma 5.1 we  deduce that
$\mu_{\sigma}^{N}(A_{\sigma})=1$ for
$\sigma \in \Sigma$ which implies

\begin{equation*}
\mu_{\sigma}^N (\{ (x_k)_{k \in N} \in {\bf R}^N
~:~T^{(1)}_{\theta_0}(x_k)_{k \in N}=\sigma \}) \ge
\end{equation*}
\begin{equation*}
\mu_{\sigma}^N (\{ (x_k)_{k \in N} \in {\bf R}^N
~:~\overline{\lim}\widetilde{T^{(1)}_n}(x_k)_{k \in N}=\sigma \}) \ge
\end{equation*}
\begin{equation*}
\mu_{\sigma}^N (\{(x_k)_{k \in N} \in {\bf  R}^N
~:~~\overline{\lim} \widetilde{T^{(1)}_n}(x_k)_{k \in N}=
\underline{\lim}\widetilde{T^{(1)}_n}(x_k)_{k \in N}=\sigma \}) =
\end{equation*}
\begin{equation*}
\mu_{\sigma}^N (\{(x_k)_{k \in N} \in {\bf  R}^N
~:~~\lim_{n \to \infty}\widetilde{T^{(1)}_n}((x_k)_{k \in N})=\sigma \}) =
\end{equation*}
\begin{equation*}
\mu_{\sigma}^N (\{(x_k)_{k \in N} \in {\bf  R}^N
~:~~\lim_{n \to \infty} \frac{1}{\sqrt{2 \pi} (n a_n)^{-1}\sum_{i=1}^n\phi(\frac{x_i-a}{a_n})}
=\sigma \}) =
\end{equation*}

\begin{equation*}
\mu_{\sigma}^N (\{(x_k)_{k \in N} \in {\bf  R}^N
~:~~\lim_{n \to \infty}(na_n)^{-1}\sum_{i=1}^n\phi(\frac{x_i-a}{a_n})=\frac{1}{\sqrt{2 \pi}\sigma}
 \}) =
\end{equation*}

\begin{equation}
\mu_{\sigma}^N (\{(x_k)_{k \in N} \in {\bf  R}^N
~:~~\lim_{n \to \infty}(na_n)^{-1}\sum_{i=1}^n\phi(\frac{x_i-a}{a_n})=f_{\sigma}(a) \}) =\mu_{\sigma}^N(A_{\sigma})=1.\label{ccs}
\end{equation}

\end{proof}

The following theorem gives  a construction of the objective infinite sample consistent estimate of an unknown parameter $\sigma$ in the same model.

\begin{thm}Let  $X_1,X_2, \cdots$  be independent normally  distributed real-valued random  variables with parameters $(a,\sigma)$, where  $a$ is a mean and $\sigma$ is a standard deviation. Suppose that we know the  mean $a$ is non-zero. Let  $\Phi$ be a standard Gaussian distribution function in ${\bf R}$.
We denote by $\mu_{\sigma}$ a Borel Gaussian probability measure in ${ \bf R}$ with parameters  $(a,\sigma)$ for each $ \sigma \in \Sigma=(0,\infty)$.
Let fix $\sigma_0 \in \Sigma$. Let define an estimate $T^{(2)}_{\sigma_0}: {\bf R}^N \to \Sigma$ as follows:
$T^{(2)}_{\sigma_0}((x_k)_{k \in N})=\overline{\lim}\widetilde{T^{(2)}_n}((x_k)_{k \in N}) $
if
$\overline{\lim}\widetilde{T^{(2)}_n}((x_k)_{k \in N}) \in \Sigma \setminus \{\sigma_0\}$ and
$T^{(2)}_{\sigma_0}((x_k)_{k \in N})=\sigma_0,$ otherwise, where
$\overline{\lim}\widetilde{T^{(2)}_n} := \inf_n \sup_{m \ge
n}\widetilde{T^{(2)}_m}$
and
\begin{equation}
\widetilde{T^{(2)}_n}((x_k)_{k \in N})=T^{(2)}_n(x_1, \cdots, x_n)= -\frac{a}{\Phi^{-1}\big(\frac{\#(\{x_1, \cdots, x_n\}\cap (-\infty,0])}{n}\big)} \label{ccs}
\end{equation}
for  $(x_k)_{k \in N} \in {\bf  R}^N $.  Then  $T^{(2)}_{\sigma_0}$  is an objective infinite sample consistent
estimator of a parameter $\sigma$ for the family
$(\mu_{\sigma}^N)_{\sigma \in \Sigma}$.
\end{thm}
%\noindent{\bf  Proof.}
\begin{proof}
Following \cite{Shiryaev80}(see, p. 189),
the function $ \overline{\lim}\widetilde{T^{(2)}_n}$  is Borel measurable which implies that the function $ \overline{\lim}\widetilde{T^{(2)}_n}$ is  $(\mathcal{B}({\bf R}^N),L(\Sigma))$-measurable.

For each $\sigma \in \Sigma$ we put
\begin{equation}
A_{\sigma}=\{ (x_k)_{k \in N} \in {\bf R}^N:(x_k)_{k \in N}~ \hbox{is ~} \mu_{\sigma}-\hbox{equidistributed~in~}{\bf R} \}.\label{ccs}
\end{equation}
By Corollary  2.16 we know that  $\mu_{\sigma}^{N}(A_{\sigma})=1$ for
$\sigma \in \Sigma$ which implies

\begin{equation*}
\mu_{\sigma}^N (\{ (x_k)_{k \in N} \in {\bf R}^N
~:~\overline{\lim}\widetilde{T^{(2)}_n}(x_k)_{k \in N}=\sigma \}) \ge
\end{equation*}
\begin{equation*}
\mu_{\sigma}^N (\{(x_k)_{k \in N} \in {\bf  R}^N
~:~~\overline{\lim} \widetilde{T^{(2)}_n}(x_k)_{k \in N}=
\underline{\lim}\widetilde{T^{(2)}_n}(x_k)_{k \in N}=\sigma \}) =
\end{equation*}
\begin{equation*}
\mu_{\sigma}^N (\{(x_k)_{k \in N} \in {\bf  R}^N
~:~~\lim_{n \to \infty}\widetilde{T^{(2)}_n}((x_k)_{k \in N})=\sigma \}) =
\end{equation*}

\begin{equation*}
\mu_{\sigma}^N (\{ (x_k)_{k \in N} \in {\bf R}^N
~:~\lim_{n \to \infty} -\frac{a}{\Phi^{-1}(\frac{(\#(\{x_1, \cdots, x_n\}\cap (-\infty,0]))}{n})}=\sigma \})=
\end{equation*}

\begin{equation*}
\mu_{\sigma}^N (\{ (x_k)_{k \in N} \in {\bf R}^N
~:~\lim_{n \to \infty} \Phi^{-1}(\frac{(\#(\{x_1, \cdots, x_n\}\cap (-\infty,0])}{n})=-\frac{a}{\sigma} \})=
\end{equation*}

\begin{equation*}
\mu_{\sigma}^N (\{ (x_k)_{k \in N} \in {\bf R}^N
~:~\lim_{n \to \infty} \frac{(\#(\{x_1, \cdots, x_n\}\cap (-\infty,0])}{n}=\Phi(-\frac{a}{\sigma}) \})=
\end{equation*}
\begin{equation*}
\mu_{\sigma}^N (\{ (x_k)_{k \in N} \in {\bf R}^N
~:~\lim_{n \to \infty} \frac{(\#(\{x_1, \cdots, x_n\}\cap (-\infty,0])}{n}=\Phi_{(a,\sigma)}(0) \}) \ge
\end{equation*}
\begin{equation}
\mu_{\sigma}^N (A_{\sigma})=1.\label{ccs}
\end{equation}
The latter relation means that $\overline{\lim}\widetilde{T^{(2)}_n}$ is a infinite-sample consistent estimate of a parameter $\sigma$ for the family  of measures $(\mu_{\sigma}^N)_{\sigma>0}$.

Let show that $\overline{\lim}\widetilde{T^{(2)}_n}$ is objective.

We have to show that for each $\sigma>0$ the set $(\overline{\lim}\widetilde{T^{(2)}_n})^{-1}(\sigma)$ is a Haar ambivalent set.

Let $(x_k)_{k \in N}$ be $\mu_{\sigma}$-equidistributed sequence.  Then we get

\begin{equation}
\lim_{n \to \infty} \frac{(\#(\{x_1, \cdots, x_n\}\cap (-\infty,0])}{n}=\Phi_{(a,\sigma)}(0) \label{ccs}
\end{equation}
which means
\begin{equation}
T^{(2)}_{\sigma_0}((x_k)_{k \in N})=\overline{\lim}\widetilde{T^{(2)}_n}((x_k)_{k \in N})=\sigma.\label{ccs}
\end{equation}
Setting $J_{\sigma}=\{ i : x_i \le 0\}$, it is not hard to show that a set
\begin{equation}
B_{J_{\sigma}}=\{ (y_i)_{i \in N} : y_i \le x_i ~\hbox{for ~}i \in J~\& y_i > x_i ~\hbox{for ~}i \in N \setminus J\} \label{ccs}
\end{equation}
is a Haar ambivalent set.

It is clear also that for each $(y_i)_{i \in N} \in B_{J_{\sigma}}$ we have
\begin{equation*}
T^{(2)}_{\sigma_0}((y_k)_{k \in N})=\overline{\lim}\widetilde{T^{(2)}_n}((y_k)_{k \in N})=\sigma,\label{ccs}
\end{equation*}
which implies that $B_{J_{\sigma}} \subseteq (\overline{\lim}\widetilde{T^{(2)}_n})^{-1}(\sigma)$.

Since $\{(\overline{\lim}\widetilde{T^{(2)}_n})^{-1}(\sigma): \sigma >0\} $ is a partition of the ${\bf R}^N$ and each of them contains a Haar ambivalent set $B_{J_{\sigma}}$ we deduce that
$(\overline{\lim}\widetilde{T^{(2)}_n})^{-1}(\sigma)$ is a Haar ambivalent set for each $\sigma>0$.

This ends the proof of the theorem.

\end{proof}

\begin{thm}Let  $X_1,X_2, \cdots$  be independent normally  distributed real-valued random  variables with parameters $(a,\sigma)$ where  $a$ is a mean and $\sigma$ is a standard deviation. Suppose that both parameters are unknown. Let  $\Phi$ be a standard Gaussian distribution function in ${\bf R}$.
We denote by $\mu_{\sigma}$ a Borel Gaussian probability measure in ${ \bf R}$ with parameters  $(a,\sigma)$ for each $ \sigma \in \Sigma=(0,\infty)$ and $a \in {\bf R}$.
Let fix $\sigma_0 \in \Sigma$. Let define an estimate $T^{(3)}_{\sigma_0}: {\bf R}^N \to \Sigma$ as follows:
$T^{(3)}_{\sigma_0}((x_k)_{k \in N})=\overline{\lim}\widetilde{T^{(3)}_n}((x_k)_{k \in N})$
if $\overline{\lim}\widetilde{T^{(3)}_n}((x_k)_{k \in N}) \in \Sigma \setminus \{\sigma_0\}$
and $T^{(3)}_{\sigma_0}((x_k)_{k \in N})=\sigma_0,$ otherwise, where
$\overline{\lim}\widetilde{T^{(3)}_n} := \inf_n \sup_{m \ge
n}\widetilde{T^{(3)}_m}$
and
\begin{equation}
\widetilde{T^{(3)}_n}((x_k)_{k \in N})=T^{(3)}_n(x_1, \cdots, x_n)= -\frac{\sum_{i=1}^nx_k}{n\Phi^{-1}\big(\frac{\#(\{x_1, \cdots, x_n\}\cap (-\infty,0])}{n}\big)} \label{ccs}
\end{equation}
for  $(x_k)_{k \in N} \in {\bf  R}^N $. Then $T^{(3)}_{\sigma_0}$  is an infinite sample consistent
estimator of a parameter $\sigma$ for the family
$(\mu_{\sigma}^N)_{\sigma \in \Sigma}$.
\end{thm}
%\noindent{\bf  Proof.}
\begin{proof}
Following \cite{Shiryaev80}(see, p. 189),
the function $ \overline{\lim}\widetilde{T^{(3)}_n}$  is Borel measurable which implies that the function $ \overline{\lim}\widetilde{T^{(3)}_n}$ is  $(\mathcal{B}({\bf R}^N),L(\Sigma))$-measurable.

For each $\sigma \in \Sigma$ we put
\begin{equation}
A_{\sigma}=\{ (x_k)_{k \in N} \in {\bf R}^N:(x_k)_{k \in N}~ \hbox{is ~} \mu_{\sigma}-\hbox{equidistributed~in~}{\bf R} \}\label{ccs}
\end{equation}
and
\begin{equation}
B_{\sigma}=\{ (x_k)_{k \in N} \in {\bf R}^N: \lim_{n \to \infty} \frac{\sum_{k=1}^nx_k }{n} =a \}.\label{ccs}
\end{equation}

On the one hand, by By Corollary  2.16 we know that  $\mu_{\sigma}^{N}(A_{\sigma})=1$ for
$\sigma \in \Sigma$. On the other hand , by Strong Law of Large Numbers we know that  $\mu_{\sigma}^{N}(B_{\sigma})=1$ for
$\sigma \in \Sigma$. These relations imply that
\begin{equation}
\mu_{\sigma}^{N}(A_{\sigma} \cap B_{\sigma})=1 \label{ccs}
\end{equation}
for $\sigma \in \Sigma$.

Take into account (5.22), we get
\begin{equation*}
\mu_{\sigma}^N (\{ (x_k)_{k \in N} \in {\bf R}^N
~:~\overline{\lim}\widetilde{T^{(3)}_n}(x_k)_{k \in N}=\sigma \}) \ge
\end{equation*}
\begin{equation*}
\mu_{\sigma}^N (\{(x_k)_{k \in N} \in {\bf  R}^N
~:~~\overline{\lim} \widetilde{T^{(3)}_n}(x_k)_{k \in N}=
\underline{\lim}\widetilde{T^{(3)}_n}(x_k)_{k \in N}=\sigma \}) =
\end{equation*}
\begin{equation*}
\mu_{\sigma}^N (\{(x_k)_{k \in N} \in {\bf  R}^N
~:~~\lim_{n \to \infty}\widetilde{T^{(3)}_n}((x_k)_{k \in N})=\sigma \}) =
\end{equation*}

\begin{equation*}
\mu_{\sigma}^N (\{ (x_k)_{k \in N} \in {\bf R}^N
~:~\lim_{n \to \infty} -\frac{\frac{\sum_{k=1}^nx_k }{n}}{\Phi^{-1}(\frac{(\#(\{x_1, \cdots, x_n\}\cap (-\infty,0]))}{n})}=\sigma \})\ge
\end{equation*}

\begin{equation*}
\mu_{\sigma}^N (\{ (x_k)_{k \in N} \in A_{\sigma} \cap B_{\sigma}
~:~\lim_{n \to \infty} \Phi^{-1}(\frac{(\#(\{x_1, \cdots, x_n\}\cap (-\infty,0])}{n})=-\lim_{n \to \infty}\frac{\frac{\sum_{k=1}^nx_k }{n}}{\sigma} \})=
\end{equation*}

\begin{equation*}
\mu_{\sigma}^N (\{ (x_k)_{k \in N} \in A_{\sigma} \cap B_{\sigma}
~:~\lim_{n \to \infty} \Phi^{-1}(\frac{(\#(\{x_1, \cdots, x_n\}\cap (-\infty,0])}{n})=-\frac{a}{\sigma} \})=
\end{equation*}

\begin{equation*}
\mu_{\sigma}^N (\{ (x_k)_{k \in N} \in {\bf R}^N
~:~\lim_{n \to \infty} \frac{(\#(\{x_1, \cdots, x_n\}\cap (-\infty,0])}{n}=\Phi_{(a,\sigma)}(0) \})=
\end{equation*}
\begin{equation}
\mu_{\sigma}^N (A_{\sigma} \cap B_{\sigma})=1.\label{ccs}
\end{equation}
The latter relation means that $T^{(3)}_{\sigma_0}$ is an infinite-sample consistent estimate of a parameter $\sigma$ for the family  of measures $(\mu_{\sigma}^N)_{\sigma>0}$.

This ends the proof of the theorem.

\end{proof}

\begin{ex} Since a sequence of
real numbers $(\pi \times n -[\pi \times n])_{n \in N}$, where
$[\cdot]$ denotes an integer part of a real number,  is uniformly
distributed on $(0,1)$(see, \cite{KuNi74}, Example 2.1, p.17), we
claim that a simulation of a $\mu_{(3,5)}$-equidistributed
sequence $(x_n)_{n \le M}$  on $R$( $M$ is a "sufficiently large"
natural number and depends on a representation quality of the
irrational number  $\pi$), where $\mu_{(3,5)}$ denotes a
a linear Gaussian measure with parameters $(3,5)$,  can
be obtained  by the formula
\begin{equation}
x_n=\Phi_{(3,5)}^{-1} (\pi \times n -[\pi \times n])\label{ccs}
\end{equation}
for $n \le M$, where $\Phi_{(3,5)}$
denotes a  Gaussian  distribution function with parameters $(3,5)$.

 Suppose that we know a mean $a=3$ and want to estimate an "unknown" standard
deviation $\sigma$.

 We set: $n$ - the number of trials;~$S_n$ - a square root from the sample variance;~$S^{'}_n$ - a square root from the corrected sample variance;
~$T^{(2)}_n$ - an estimate defined by the formula (5.13); ~$T^{(3)}_n$ - an estimate defined by the formula (5.19);
~$\sigma$ - an unknown standard deviation.

 The numerical data placed in Table 3 were obtained by using  Microsoft Excel.
\begin{table*}
\caption{Estimates of an unknown standard deviation $\sigma=5$ }
\label{sphericcase}
\begin{tabular}{crrrc}
\hline
 \multicolumn{1}{c}{$n$} & \multicolumn{1}{c}{$S_n$} & \multicolumn{1}{c}{$S^{'}_n$}  & \multicolumn{1}{c}{$T^{(2)}_n$} & \multicolumn{1}{c}{$T^{(3)}_n$}
  \\
\hline
 $ 200 $       &    $4.992413159$  & $5.004941192$  & $5.205401325
$ &     $4.895457577$  \\

$ 400 $       &    $4.992413159$  & $5.004941192$  & $5.141812921
$ &     $4.835655399$  \\

$ 600 $       &    $5.10523925$  & $5.109498942$  & $5.211046737
$ &     $4.855457413$ \\

$ 800 $       &    $5.106390271$  & $5.109584761$  & $5.19369988
$ &     $4.92581015$ \\
$ 1000 $       &    $5.066642282$  & $5.069177505$  & $5.028142523
$ &     $4.944169095$ \\

$ 1200 $       &    $5.072294934$  & $5.074409712$  & $5.235885276
$ &     $4.935995814$ \\

$ 1400 $       &    $5.081110418$  & $5.082926073$  & $5.249446371
$ &     $4.96528786$ \\

$ 1600 $       &    $5.079219075$  & $5.080807075$  & $5.205452797
$ &     $4.9564705$ \\

$ 1800 $       &    $5.060850283$  & $5.06225666$  & $5.207913228
$ &     $4.963326232$ \\

$ 2000 $       &    $5.063112113$  & $5.064378366$  & $5.239119585
$ &   $4.981223889$ \\
\end{tabular}
\end{table*}
Notice  that results of computations presented in Table 3 show us that both statistics  $T^{(2)}_n$ and $T^{(3)}_n$ work correctly.
Unfortunately, we can not present the results  of computations for the statistics  $T^{(1)}_n$ defined by the formula  (5.10) because it needs a quite exact calculations.

\end{ex}

At end of this section we state the following
\begin{prob} Let  $\mathcal{D}$ be a  class  of positive continuous probability densities  and $p_{f}$ be a Borel probability measure
on ${\bf R}$ with probability density function $f$ for each $f \in \mathcal{D}$. Does there exist an objective (or a subjective)  infinite sample consistent estimate of an unknown probability density function
 $f$ for the  family of Borel probability measures  $\{ p_f^N : f \in \mathcal{D})\} $?
\end{prob}

\section{On orthogonal statistical structures in a non-locally-compact Polish
group admitting an invariant metric}

%\vspace*{0.3cm}{\bf  Theorem 4.1}~{\it
Let $G$ be a Polish group, by which we mean a separable group with a
complete invariant metric $\rho$ (i.e., $\rho(fh_1g,fh_2g)=\rho(h_1,h_2)$ for each $f,g,h_1,h_2 \in G$) for which the transformation (from $G \times G$ onto $G$),
which sends $(x, y)$ into $x^{-1}y$ is continuous. Let $\mathcal{B}(G)$ denotes the
$\sigma$-algebra of Borel subsets of $G$.
\begin{defn} \cite{Myc92} A Borel set $X \subseteq G$ is called shy, if
there exists a Borel probability measure $\mu$ over $G$  such that  $\mu(f X g) = 0$
 for all $f, g  \in G.$  A measure $\mu$ is called a testing measure for a set $X$.
  A subset of a Borel shy set is
called also shy. The complement of a shy set is
called a prevalent set.
\end{defn}
\begin{defn} \cite{Balka12} A Borel set is called a Haar ambivalent set if it is neither shy nor
prevalent.
\end{defn}
\begin{rem}Notice that  if $X \subseteq G$ is shy then there exists such a testing measure $\mu$ for a set $X$  which has  a compact carrier $K \subseteq G$(i.e.
$\mu(G \setminus K)=0$). The collection of shy  sets constitutes  an $\sigma$-ideal, and in the case where $G$ is locally
compact a  set is shy iff it has Haar measure zero.
\end{rem}
\begin{defn} If $G$ is a Polish group and $\{\mu_{\theta}: \theta \in \Theta\}$ is  a family of Borel probability measures on  $G$, then the family of triplets
$\{(G, \mathcal{B},\mu_{\theta}): \theta \in \Theta\}$, where $\Theta$ is a non-empty set equipped with an $\sigma$ algebra $L(\Theta)$ generated by  all singletons of $\Theta$,   is called a statistical structure. A set $\Theta$ is called a set of parameters.
\end{defn}
\begin{defn}$(\mathcal{O})$ The statistical structure $\{(G, \mathcal{B}(G),\mu_{\theta}): \theta \in \Theta\}$ is called orthogonal
if the measures $\mu_{\theta_1}$ and $\mu_{\theta_2}$ are  orthogonal for each different parameters  $ \theta_1$ and $\theta_2$.
\end{defn}
\begin{defn} $(\mathcal{WS})$ The statistical structure $\{(G, \mathcal{B}(G),\mu_{\theta}): \theta \in \Theta\}$ is called weakly separated
if there exists a family of Borel subsets $\{X_\theta:\theta \in \Theta\}$ such that $\mu_{\theta_1}(X_{\theta_2})=\delta(\theta_1,\theta_2)$, where $\delta$ denotes Kronecker’s function defined on the Cartesian square $\Theta \times \Theta$ of the set $\Theta$. \end{defn}
\begin{defn}$(\mathcal{SS})$ The statistical structure $\{(G, \mathcal{B}(G),\mu_{\theta}): \theta \in \Theta\}$ is called strong separated (or strictly separated)
if there exists a partition of the group $G$ into family of Borel subsets $\{X_\theta:\theta \in \Theta\}$ such that $\mu_{\theta}(X_{\theta})=1$ for each $\theta \in \Theta$. \end{defn}
\begin{defn}$(\mathcal{CE})$  A $(\mathcal{B}(G), L(\Theta))$-measurable mapping  $T : G \to \Theta$ is called a consistent estimate of an unknown parameter $\theta \in \Theta$ for the statistical structure $\{(G, \mathcal{B}(G),\mu_{\theta}): \theta \in \Theta\}$ if the condition
 $\mu_{\theta}(T^{-1}(\theta))=1$ holds true for each $\theta \in \Theta$.
 \end{defn}
 \begin{defn} $(\mathcal{OCE})$  A $(\mathcal{B}(G), L(\Theta))$-measurable mapping  $T : G \to \Theta$ is called an objective  consistent estimate of an unknown  parameter $\theta \in \Theta$ for the statistical
 structure $\{(G, \mathcal{B}(G),\mu_{\theta}): \theta \in \Theta\}$  if the following two conditions hold:

 (i)~ $\mu_{\theta}(T^{-1}(\theta))=1$ for each $\theta \in \Theta$;

 (ii)~ $T^{-1}(\theta)$ is a Haar ambivalent set for each $\theta \in \Theta$.

 If the condition (i) holds but the condition (ii) fails, then $T$ is called  a subjective  consistent estimate of an unknown  parameter $\theta \in \Theta$ for the statistical
 structure $\{(G, \mathcal{B},\mu_{\theta}): \theta \in \Theta\}$.
 \end{defn}
\begin{defn}$(\mathcal{SOCE})$  An objective  consistent estimate $T : G \to \Theta$  of an unknown  parameter $\theta \in \Theta$ for the statistical
 structure $\{(G, \mathcal{B}(G),\mu_{\theta}): \theta \in \Theta\}$  is called {\it strong} if for each $\theta_1, \theta_2 \in \Theta $ there exists an isometric Borel measurable bijection $A_{(\theta_1,\theta_2)}:G \to G$ such that the set $A_{(\theta_1,\theta_2)} (T^{-1}(\theta_1)) \Delta T^{-1}(\theta_2)$ is shy in $G$.
 \end{defn}
\begin{rem} Let $G$ be a Polish
non-locally-compact group admitting an invariant metric. The relations between  statistical structures introduced in Definitions 6.5-6.10  for such a group  can be presented by  the following diagram:
\begin{equation}
\begin{matrix}
\mathcal{SOCE}& \rightarrow&  \mathcal{OCE}&\rightarrow&  \mathcal{CE} \leftrightarrow \mathcal{SS}&  \rightarrow&  \mathcal{WS}&  \rightarrow &  \mathcal{O}\\
\end{matrix}
\label{ccs}
\end{equation}
\end{rem}

To show that the converse implications  sometimes fail  we consider the following examples.

\begin{ex}$\rceil(\mathcal{WS} \leftarrow  \mathcal{O})$~~Let $F \subset G$ be a closed subset of the cardinality $2^{\aleph_0}$. Let $\phi:[0,1]\to F$ be a Borel isomorphism of $[0,1]$ onto $F$. We set $\mu(X)=\lambda(\phi^{-1}(X \cap F))$  for $X \in \mathcal{B}(G)$, where $\lambda$ denotes a linear Lebesgue measure on $[0,1]$. We put $\Theta=F$. Let fix $\theta_0 \in \Theta$ and put: $\mu_{\theta}=\mu$ if $\theta=\theta_0$, and $\mu_{\theta}=\delta_{\theta}|_{\mathcal{B}(G)}$, otherwise, where $\delta_{\theta}$ denotes a Dirac measure on $G$  concentrated at the point $\theta$  and  $\delta_{\theta}|_{\mathcal{B}(G)}$ denotes the restriction of the  $\delta_{\theta}$  to the class $\mathcal{B}(G)$. Then the statistical
 structure $\{(G, \mathcal{B},\mu_{\theta}): \theta \in \Theta\}$ stands  $\mathcal{O}$ which is not $\mathcal{WS}$.
\end{ex}

\begin{ex}(SM) $\rceil(\mathcal{SS} \leftarrow  \mathcal{WS})$  Following \cite{Pan03}(see, Theorem 1, p. 335), in the system of axioms (ZFC) the following three conditions
are equivalent:

1) The Continuum Hypothesis $(c = 2^{\aleph_0}= \aleph_1)$;

2) for an arbitrary probability space $(E; S; \mu)$, the $\mu$-measure of the union of
any family $(E_i)_{i \in I}$  of $\mu$-measure zero subsets, such that $\mbox{card}(I) < c$, is equal to
zero;

3) an arbitrary weakly separated family of probability measures, of cardinality
continuum, is strictly separated.

 The latter relation means that   under  Continuum Hypothesis in $ZFC$  we have $\mathcal{SS} \leftarrow  \mathcal{WS}$. This is just Skorohod well known result(see, \cite{IbramSkor80}).  Moreover, following \cite{Pan03}(see Theorem 2, p.339), if $(F,\rho)$ is a Radon metric space and $(\mu_i)_{i \in I}$ is a weakly
separated family of Borel probability measures with $\mbox{card}(I)\le c$, then  in the system of axioms $(ZFC) \& (MA)$, the family $(\mu_i)_{i \in I}$ is strictly
separated.

Let consider a counter example to the implication  $\mathcal{SS} \leftarrow  \mathcal{WS}$ in the Solovay model (SM) \cite{Solovay92} which is the following system of axioms:
$(ZF)$+$DC$+ ''every subset of the real axis ${\bf R}$ is Lebesgue measurable'', where  $(ZF)$ denotes the  Zermelo-Fraenkel set theory and $(DC)$
denotes the axiom of Dependent Choices.

For $\theta \in (0;1)$ , let $b_{\theta}$ be a linear classical Borel measure defined on the set
$\{\theta\} \times (0;1)$. For $\theta \in  (1.2)$, let  $b_{\theta}$ be a linear classical Borel measure defined on the set
$(0;1) \times  \{\theta-1\}$. By $\lambda_{\theta}$ we denote a Borel probability measure on $(0;1)\times (0;1)$ produced by $b_{\theta}$, i.e.,
\begin{equation*}
(\forall X)(\forall \theta_1)(\forall \theta_2)(X \in \mathcal{B}((0;1)\times (0;1))~\&~\theta_1 \in (0;1)~\&~\theta_2 \in (1;2) \rightarrow
\end{equation*}
\begin{equation}
\lambda_{\theta_1}(X)=b_{\theta_1}((\{\theta_1\} \times (0;1)) \cap X)~\&~ \lambda_{\theta_2}(X)=b_{\theta_2}(((0;1) \times \{\theta_1-1\}) \cap X)).
\label{ccs}
\end{equation}
If we put $\theta=(0;1) \cup (1;2)$, then we get a statistical structure
\begin{equation}((0;1)\times (0;1), \mathcal{B}((0;1)\times (0;1)), \lambda_{\theta})_{\theta \in \Theta}.\label{ccs}
\end{equation}
Setting  $X_\theta=\{\theta\} \times (0;1)$ for $\theta \in (0;1)$, and  $X_\theta=(0;1) \times  \{\theta-1\}$ for $\theta \in  (1.2)$, we observe that for the family of Borel subsets $\{X_\theta:\theta \in \Theta\}$ we have $\lambda_{\theta_1}(X_{\theta_2})=\delta(\theta_1,\theta_2)$, where $\delta$ denotes Kronecker’s function defined on the Cartesian square $\Theta \times \Theta$ of the set $\Theta$. In other words,  $(\lambda_{\theta})_{\theta \in \Theta}$ is weakly separated. Now let assume that this family is strong separated. Then there
will be a partition $\{ Y_{\theta}:\theta \in \Theta\} $ of the set $(0;1)\times (0;1)$ into Borel subsets $(Y_{\theta})_{\theta \in \Theta}$ such that $\lambda_{\theta}(Y_{\theta})=1$ for each $\theta \in \Theta$. If we consider $A=\cup_{\theta \in (0;1)}Y_{\theta}$ and $B=\cup_{\theta \in (1;2)}Y_{\theta}$ then we observe by Fubini theorem that $\ell_2(A)=1$ and $\ell_2(B)=1$, where $\ell_2$ denotes the $2$-dimensional Lebesgue measure defined in $(0;1)\times (0;1)$. This is the contradiction and we proved that $(\lambda_{\theta})_{\theta \in \Theta}$ is not strictly separated. An  existence of a Borel isomorphism $g$ between $(0;1)\times(0;1)$ and $G$ allows us to construct a family $(\mu_{\theta})_{\theta \in \Theta}$ in $G$ as follows: $\mu_{\theta}(X)=\lambda_{\theta}(g^{-1}(X))$ for each $X \in \mathcal{B}(G)$ and $\theta \in \Theta$ which is $\mathcal{WS}$ but no $\mathcal{SS}$(equivalently, $\mathcal{CE}$).
 By virtue the celebrated result of  Mycielski and Swierczkowski (see, \cite{Myc64}) asserted that under the Axiom of Determinacy $(AD)$ every subset of the real axis ${\bf  R}$ is Lebesgue measurable, the same
example can be used as a counter example to the implication  $\mathcal{SS} \leftarrow  \mathcal{WS}$  in the theory $(ZF)+(DC)+(AD)$.
Since the answer to the question asking ''{\it whether $(\mu_{\theta})_{\theta \in \Theta}$ has a consistent estimate?}'' is {\it \bf yes} in the theory $(ZFC)~\&~(CH)$, and {\it \bf no} in the theory   $(ZF)+(DC)+(AD)$, we deduce that this question is not solvable within the theory $(ZF)+(DC)$.
\end{ex}

\begin{ex}$\rceil(\mathcal{OCE} \leftarrow  \mathcal{CE})$  Setting $\Theta=G$ and $\mu_{\theta}=\delta_{\theta}|\mathcal{B}(G)$ for $\theta \in \Theta$, where $\delta_{\theta}$ denotes a Dirac measure in $G$ concentrated at the point $\theta$  and  $\delta_{\theta}|\mathcal{B}(G)$ denotes its restriction to $\mathcal{B}(G)$,  we get a statistical structure $(G, \mathcal{B}(G), \mu_{\theta})_{\theta \in \Theta}$. Let $L(\Theta)$ denotes a minimal $\sigma$-algebra of subsets of $\Theta$ generated by all singletons of $\Theta$.
Setting $T(g)= g$  for $g \in G$,  we get a consistent estimate of an unknown  parameter $\theta$ for the
family $(\mu_{\theta})_{\theta \in \Theta}$.  Notice that there does not exist  an objective consistent estimate of a parameter $\theta$ for the
family $(\mu_{\theta})_{\theta \in \Theta}$. Indeed, if we assume the contrary and $T_1$ be such an estimate, we get that $T_1^{-1}(\theta)$ is a Haar ambivalent set for each $\theta \in \Theta$.
Since $T_1$ is a consistent estimate of an unknown  parameter $\theta$  for each $\theta \in \Theta$, we get that the condition $\mu_{\theta}( T_1^{-1}(\theta))=1$ holds true which implies
that $\theta \in T_1^{-1}(\theta)$ for each $\theta \in \Theta$.
Let fix any parameter $\theta_0 \in \Theta$. Since $T_1^{-1}(\theta_0)$ is a Haar ambivalent set there is $\theta_1 \in T_1^{-1}(\theta_0)$ which differs from $\theta_0$. Then $T_1^{-1}(\theta_0)$ and
$T_1^{-1}(\theta_1)$ are not disjoint because  $\theta_1 \in T_1^{-1}(\theta_0) \cap T_1^{-1}(\theta_1)$ and we get the contradiction.
\end{ex}
\begin{rem} Notice that if $(\Theta, \rho)$ is a metric space and if  in the Definition 2.8 the requirement of a
$(\mathcal{B}(G), L(\Theta))$-measurability will be  replaced with a $(\mathcal{B}(G), \mathcal{B}(\Theta))$-measurability, then the implication  $\mathcal{SS} \rightarrow    \mathcal{CE}$ may be false. Indeed, let $G$ be a Polish group and   $f: G \leftarrow \Theta(:= G)$  be a non-measurable(in the Borel sense) bijection. For each $\theta \in \Theta$ denote by $\mu_{\theta}$ the restriction of the Dirac measure $\delta_{f(\theta)}$ to the $\sigma$-algebra of  Borel subsets of the group $G$. It is clear that the statistical structure $\{(G, \mathcal{B}(G), \mu_{\theta}): \theta \in \Theta\}$ is strictly separated. Let show that there does not exist a consistent estimate for that statistical structure. Indeed, let $T:G \to \Theta$ be $(\mathcal{B}(G), \mathcal{B}(\Theta))$-measurable mapping such that
$\mu_{\theta}(\{ x: T(x)=\theta\})=1$ for each $\theta \in \Theta$. Since the measure $\mu_{\theta}$ is concentrated at the point $f(\theta)$ we have that $f(\theta) \in \{ x:T(x)=\theta\}$   for each $\theta \in \Theta$ which implies that  $T(f(\theta))=\theta$ for each $\theta \in \Theta$. The latter relation means that $T=f^{-1}$. Since $f$ is not $(\mathcal{B}(G), \mathcal{B}(\Theta))$-measurable, we claim that $f^{-1}=T$ is not also $(\mathcal{B}(G), \mathcal{B}(\Theta))$-measurable and we get the contradiction.
\end{rem}
There naturally arises a question asking whether there exists such a statistical
 structure $\{(G, \mathcal{B},\mu_{\theta}): \theta \in \Theta\}$ in a Polish non-locally-compact group admitting
an invariant metric which has an objective consistent estimate of a parameter $\theta$. To answer positively to this question,  we  need the following two lemmas.
\begin{lem} (\cite{Solecki96}, Theorem, p.206)  Assume $G$ is a Polish, non-locally-compact group admitting
an invariant metric. Then there exists a closed set $F\subseteq G$ and a continuous
function  $ \phi: F \to 2^N$  such that for any $x  \in 2^N$ and any compact set $K \subseteq G$
there is $g \in G$ with $g K \subseteq \phi^{-1}(x)$.
\end{lem}
\begin{lem} (\cite{Dough94}, Proposition 12, p.87 ). Let $G$ be a non-locally-compact Polish group with an
invariant metric. Then any compact subset (and hence any $K_{\sigma}$ subset) of
$G$ is shy.
\end{lem}
\begin{rem} In \cite{Pan14-1}(see proof of Theorem 4.1, Step 2) has been constructed a partition $\Phi=\{ A_{\theta} : \theta \in [0,1] \}$ of the ${\bf R}^N$ into Haar ambivalent sets such that  for each $\theta_1,\theta_2 \in [0,1]$ there exists  an isometric( with respect to Tychonoff metric which is invariant under translates) Borel measurable bijection $A_{(\theta_1,\theta_2)}$   of ${ \bf R}^N$ such that
$A_{(\theta_1,\theta_2)}(A_{\theta_1}) \Delta A_{\theta_2}$ is shy. In this context and concerning with Lemma 6.16 it is natural to ask  whether an arbitrary  Polish non-locally-compact group with an invariant metric admits a similar partition into  Haar ambivalent sets.  Notice that we have no any information in this direction.
\end{rem}
\begin{thm}  Let  $G$ be a Polish non-locally-compact group admitting
an invariant metric. Then there exists  a statistical
 structure $\{(G, \mathcal{B},\mu_{\theta}): \theta \in \Theta\}$ in $G$  which has an objective consistent estimate of a parameter $\theta$ such that:

 (i)~$\Theta \subseteq G$ and  $\mbox{card}(\Theta)=2^{\aleph_0}$;

 (ii)~ $\mu_{\theta}$ is the restriction of the  Dirac measure concentrated at the point $\theta$ to the Borel $\sigma$-algebra  $\mathcal{B}(G)$  for each $\theta \in \Theta$.
\end{thm}
\begin{proof} By virtue of Lemma 6.16, there exists a closed set $F\subseteq G$ and a continuous
function  $ \phi: F \to 2^N$  such that for any $x  \in 2^N$ and any compact set $K \subseteq G$
there is $g \in G$ with $g K \subseteq \phi^{-1}(x)$.
For $x \in 2^N \setminus \{(0,0,\cdots)\}$ we put
\begin{equation}
X_{x}=\phi^{-1}(x).
\label{ccs}
\end{equation}
We set $X_{(0,0,\cdots)}=\phi^{-1}((0,0,\cdots)) \cup (G \setminus F).$
Thus we have a partition $\{ X_{x}: x \in 2^N \}$ of $G$ into Borel subsets   such that each element of the partition is Borel measurable and  Haar ambivalent set. Let $\{ \theta_x : x \in 2^N  \}$ be any selector. We put $\Theta=\{\theta : \theta=\theta_x~\mbox{  for ~some}~  x \in 2^N\}$ and denote by $\mu_{\theta}$ the restriction of the Dirac measure concentrated at the point $\theta$ to the $\sigma$-algebra $\mathcal{B}(G)$. Thus we have constructed  a statistical
 structure $\{(G, \mathcal{B},\mu_{\theta}): \theta \in \Theta\}$ in $G$. We put $T(g)=\theta$  for each $g \in X_{\theta}$. Now it is obvious that
 $T$ is the objective consistent estimate of a parameter $\theta$ for the statistical structure $\{(G, \mathcal{B},\mu_{\theta}): \theta \in \Theta\}$ in $G$ such that the conditions (i)-(ii) are fulfilled.
 \end{proof}

\begin{thm}  Let $G$ be a Polish non-locally-compact group admitting
an invariant metric. Let $\mu$ be a Borel probability measure whose carrier is a compact set $K_0$( i.e., $\mu(G \setminus K_0)=0$).  Then there exists  a statistical
 structure $\{(G, \mathcal{B},\mu_{\theta}): \theta \in \Theta\}$ in $G$  which has an objective consistent estimate of a parameter $\theta$ such that

 (i)~$\Theta \subseteq G$ and  $\mbox{card}(\Theta)=2^{\aleph_0}$;

 (ii)~ $\mu_{\theta}$ is a $\theta$-shift of the measure $\mu$ (i.e. $\mu_{\theta}(X)=\mu(\theta^{-1} X)$ for $X \in \mathcal{B}(G)$ and $\theta \in \Theta$).
\end{thm}
\begin{proof} By virtue of Lemma 6.16, there exists a closed set $F \subseteq G$ and a continuous
function  $ \phi: F \to 2^N$  such that for any $x  \in 2^N$ and any compact set $K \subseteq G$
there is $g \in G$ with $g K \subseteq \phi^{-1}(x)$.
For $x \in 2^N \setminus \{(0,0,\cdots)\}$ we put
$X_{x}=\phi^{-1}(x)$. We set $X_{(0,0,\cdots)}=\phi^{-1}((0,0,\cdots)) \cup (G \setminus F)$. Thus we have a partition $\{ X_{x}: x \in 2^N \}$ of $G$ into Borel subsets   such that each element of the partition is Borel measurable, Haar ambivalent set  and for any $x  \in 2^N$ and any compact set $K \subseteq G$
there is $g \in G$ with $g K \subseteq X_{x}$.
If we take under $K$ a set $K_0$, then for any $x  \in 2^N$ there is $g(K_0,x) \in G$ with $g(K_0,x)K_0 \subseteq X_x$.
We put $\Theta=\{\theta: \theta=g(K_0,x)~\&~x \in 2^N\}$. For each $\theta \in \Theta$ and $X \in \mathcal{B}(G)$, we put $\mu_{\theta}(X)=\mu(\theta^{-1} X)$.
For $g \in X_{x}$ we put $T(g)=g(K_0,x)$. Let us show that $T:G \to \Theta$ is an objective  consistent estimate of a parameter $\theta$. Indeed, on the one hand, for each $\theta \in \Theta$ we have
\begin{equation*}
  \mu_{\theta}(T^{-1}(\theta))=\mu_{g(K_0,x)}(T^{-1}(g(K_0,x)))=\mu_{g(K_0,x)}(X_{x})=\mu(g(K_0,x)^{-1} X_{x})\ge
  \end{equation*}
  \begin{equation}
  \mu(g(K_0,x)^{-1}g(K_0,x)K_0)=\mu(K_0)=1,\label{ccs}
\end{equation}
  which means that $T:G \to \Theta$ is a consistent estimate of a parameter $\theta$.
On the other hand, for each $\theta=g(K_0,x) \in \Theta$ we have that a set $T^{-1}(\theta)=T^{-1}(g(K_0,x))=X_x$ is Borel measurable and a Haar ambivalent set which together with the latter relation implies that $T:G \to \Theta$ is an objective  consistent estimate of a parameter $\theta$. Now it is obvious to check that for the statistical structure $\{(G, \mathcal{B},\mu_{\theta}): \theta \in \Theta\}$  the conditions (i)-(ii) are fulfilled.
 \end{proof}
The next theorem shows whether  can be constructed  an objective consistent estimates by virtue of some consistent estimates in a Polish non-locally-compact group admitting
an invariant metric.
\begin{thm}  Assume $G$ is a Polish, non-locally-compact group admitting
an invariant metric. Let $\mbox{card}(\Theta)=2^{\aleph_0}$ and $T : G \to \Theta$ be  a  consistent estimate of a parameter $\theta$ for the
family of Borel probability measures $(\mu_{\theta})_{\theta \in \Theta}$ such that there exists $\theta_0 \in \Theta$ for which
$T^{-1}(\theta_0)$ is a prevalent set. Then there exists an objective consistent estimate of a parameter $\theta$ for the
family $(\mu_{\theta})_{\theta \in \Theta}$.
\end{thm}
\begin{proof} For $\theta \in \Theta$ we put $S_{\theta}=T^{-1}(\theta)$. Since $S_{\theta_0}$ is a prevalent set
we deduce that
\begin{equation}
\cup_{ \theta \in \Theta \setminus \{\theta_0\}}S_{\theta}={\bf  R}^N \setminus S_{\theta_0}\label{ccs}
\end{equation} is shy in $G$.

By Lemma 2.6  we know that the measure $\mu_{\theta_0}$ is concentrated on a union of a countable family of compact subsets $\{F_k^{(\theta_0)}:k \in N\}$. By Lemma 6.7 we know that $\cup_{k \in N}F_k^{(\theta_0)}$ is shy in $G$.

We put $\tilde{S}_{\theta}=S_{\theta}$ for $\theta \in \Theta \setminus \{\theta_0\}$  and $\tilde{S}_{\theta_0}=\cup_{k \in N}F_k^{(\theta_0}$.
Clearly, $S=\cup_{\theta \in \Theta}\tilde{S}_{\theta}$ is shy in $G$.

By virtue of Lemma 6.16, there exists a closed set $F \subseteq G$ and a continuous
function  $ \phi: F \to 2^N$  such that for any $x  \in 2^N$ and any compact set $K \subseteq G$
there is $g \in G$ with $g K \subseteq \phi^{-1}(x)$. Let $f: 2^N \to \Theta$ be any bijection. For $\theta \in \Theta$ we put

\begin{equation}
B_\theta= ( \phi^{-1}(f^{-1}(\theta)) \setminus S) \cup S_{\theta}.
\label{ccs}
\end{equation}
 Notice that $(B_\theta)_{\theta \in \Theta}$ is a partition of $G$ into Haar ambivalent sets.
 We put  $T_1(g)=\theta$ for $g \in B_{\theta} (\theta \in \Theta)$.
 Since
\begin{equation}
\mu_{\theta}(T_1^{-1}(\theta))=\mu_{\theta}(B_{\theta}) \ge \mu_{\theta}(S_{\theta})=1\label{ccs}
\end{equation}
 for $\theta \in \Theta$, we claim that $T_1$ is a consistent estimate of a parameter $\theta$ for the
family $(\mu_{\theta})_{\theta \in \Theta}$.
 Since $T_1^{-1}(\theta))=B_{\theta}$ is a Borel and Haar ambivalent set for each $\theta \in \Theta$ we end the proof of the theorem.
\end{proof}

\begin{ex} Let $F$ be a distribution function on ${\bf  R}$  such that the integral $\int_{{\bf  R}}xdF(x)$ exists and is equal to zero.
Suppose that $p$ is a Borel probability measure on ${\bf  R}$ defined by $F$. For $\theta \in \Theta(:={\bf  R})$, let $p_{\theta}$ be $\theta$-shift of the measure $p$(i.e., $p_{\theta}(X)=p(X-\theta)$ for $X \in \mathcal{B}({\bf  R})$). Setting, $G={\bf  R}^N$, for $\theta \in \Theta$ we put $\mu_{\theta}=p_{\theta}^N$, where $p_{\theta}^N$ denotes the infinite
power of the measure $p_{\theta}$. We set $T((x_k)_{k \in N})=\lim_{n \to \infty}\frac{\sum_{k=1}^nx_k}{n}$, if $\lim_{n \to \infty}\frac{\sum_{k=1}^nx_k}{n}$ exists, is finite and differs from the zero, and
 $T((x_k)_{k \in N})=0$, otherwise. Notice that  $T : {\bf  R}^N \to \Theta$ is a  consistent estimate of a parameter $\theta$ for the
family $(\mu_{\theta})_{\theta \in \Theta}$ such that
$T^{-1}(0)$ is a prevalent set. Indeed, by virtue the Strong Law of Large Numbers, we know that
\begin{equation}
\mu_{\theta}(\{ (x_k)_{k \in N}:\lim_{n \to \infty}\frac{\sum_{k=1}^nx_k}{n}=\theta\}=1
\label{ccs}
\end{equation}
for $\theta \in \Theta$.

Following  \cite{Pan14}(Lemma 4.14, p. 60), a set $S$ defined by

\begin{equation}
S=\{ (x_k)_{k \in N}:\lim_{n \to \infty}\frac{\sum_{k=1}^nx_k}{n} ~\mbox{exists~and ~is~ finite} \},\label{ccs}
\end{equation}
 is Borel shy set, which implies that ${\bf  R}^N \setminus S$ is a prevalent set. Since ${\bf  R}^N \setminus S  \subseteq  T^{-1}(0)$, we deduce that
$T^{-1}(0)$ is a prevalent set. Since for the statistical structure $\{({\bf  R}^N, \mathcal{B}({\bf  R}^N),\mu_{\theta}): \theta \in \Theta\}$ all conditions of the Theorem 6.21 are fulfilled, we claim that there exists  an objective consistent estimate of a parameter $\theta$ for the
family $(\mu_{\theta})_{\theta \in \Theta}$.
\end{ex}

Notice that in  Theorem 4.1 (see also \cite{Pan14-1}, Theorem 3.1, p. 117)  has been considered an  example of a  strong objective infinite sample consistent estimate  of an unknown parameter for a certain  statistical structure in the Polish non-locally compact abelian group ${\bf  R}^N$. In context with this example we state the following
\begin{prob}Let $G$ be a Polish  non-locally-compact group admitting
an invariant metric. Does there exist a statistical structure $\{ (G,\mathcal{B}(G),\mu_{\theta}):\theta
\in \Theta \}$ with $\mbox{card}(\Theta)=2^{\aleph_0}$ for which there exists a strong objective consistent estimate of a parameter $\theta$?
\end{prob}

\section{On objective and strong objective  consistent estimates of an unknown parameter in a compact Polish
 group $\{0,1\}^N$}

 Let   $x_1,x_2, \cdots, x_k, \cdots$  be an infinite sample obtained by coin tosses. Then the statistical structure described this experiment has the form:
\begin{equation}
\{(\{0,1\}^N,B(\{0,1\}^N),\mu_{\theta}^N): \theta \in (0,1)\}
\label{ccs}
\end{equation}
where $\mu_{\theta}(\{1\})=\theta$ and $\mu_{\theta}(\{0\})=1-\theta$.
By virtue of the Strong Law of Large Numbers we have
\begin{equation}
 \mu_{\theta}^N(\{(x_k)_{k \in N}: (x_k)_{k \in N}\in \{0,1\}^N~\&~\lim_{n \to \infty}\frac{\sum_{k=1}^nx_k}{n}=\theta\})=1
 \label{ccs}
\end{equation}
 for each $\theta \in (0,1)$.

Notice that for each $k \in N$, $G_k= \{0,1\}$ can be considered as a compact group  with an addition group operation $\pmod{2}$.
Hence the space of all infinite samples $G:=\{0,1\}^N$ can be presented as an infinite product of compact groups $\{G_k:k \in N\}$, i.e. $G=\prod_{k \in N}G_k$. Also, that the group $G$ admits an invariant metric $\rho$ which is defined by $\rho ((x_k)_{k \in N}, (y_k)_{k \in N})=\sum_{k \in N}\frac{|x_k-y_k \pmod{2}|}{2^{k+1}(1+|x_k-y_k \pmod{2}|)}$ for  $(x_k)_{k \in N}, (y_k)_{k \in N} \in G$.
It is obvious that the measure $\lambda_k$ on $G_k$ defined by $\lambda_k(\{0\})= \lambda_k(\{1\})=1/2$ is a probability Haar measure in $G_k$ for each $k \in N$ and for the probability Haar measure $\lambda$ in $G$ the following equality $\lambda=\prod_{k \in N}\lambda_k$ holds true, equivalently, $\lambda=\mu_{0,5}^N$.

By virtue (7.2) we  deduce that  the set
\begin{equation}
A(0,5)=\{(x_k)_{k \in N}: (x_k)_{k \in N}\in \{0,1\}^N~\&~
\lim_{n \to \infty}\frac{\sum_{k=1}^nx_k}{n}=0,5\}\label{ccs}
\end{equation} is prevalence.
Since $A(\theta) \subset G \setminus A(0,5)$ for each $\theta \in (0;1)\setminus \{1/2\}$, where
\begin{equation}
A(\theta)=\{(x_k)_{k \in N}: (x_k)_{k \in N}\in \{0,1\}^N~\&~\lim_{n \to \infty}\frac{\sum_{k=1}^nx_k}{n}=\theta\},\label{ccs}
\end{equation}
we deduce that  they all are shy(equivalently, of  Haar measure zero) sets.
In terms of \cite{HSY92}, this phenomena can be expressed in the following form.
\begin{thm} For "almost every"  sequence $(x_k)_{k \in N}\in \{0,1\}^N$ its Cezaro means $(\frac{\sum_{k=1}^nx_k}{n})_{n \in N}$ converges to $0,5$  whenever $n$ tends to $\infty$.
\end{thm}
By virtue the Strong Law of Large Numbers, we get

\begin{thm} Let fix $\theta_0 \in (0,1)$. For each $(x_k)_{k \in N}\in G$, we set $T((x_k)_{k \in N})= \lim_{n \to \infty}\frac{\sum_{k=1}^nx_k}{n}$ if this limit exists and differs from $\theta_0$, and  $T((x_k)_{k \in N})=\theta_0$, otherwise. Then $T$ is a consistent estimate of an unknown parameter $\theta$ for the statistical
structure $\{(G, \mathcal{B}(G), \mu_{\theta}): \theta \in \Theta \}$.
\end{thm}
\begin{rem} Following Definition 6.9, the estimate $T$ is subjective because $T^{-1}(1/2)$ is a prevalent set. Unlike  Theorem 6.21, there does not exist an objective consistent estimate of an unknown parameter $\theta$ for any  statistical
structure $\{(G, \mathcal{B}(G), \mu_{\theta}): \theta \in \Theta \}$ for which $\hbox{card}(\Theta)>\aleph_0$, where $\aleph_0$ denotes the cardinality of the set of all natural numbers. Indeed, assume the contrary and let $T_1$ be such an estimate. Then we get the partition $\{ T_1^{-1}(\theta): \theta \in \Theta \}$ of the compact group $G$ into Haar ambivalent sets. Since each Haar ambivalent set is of positive $\lambda$-measure, we get that the probability Haar measure $\lambda$ does not satisfy Suslin  property provided that the cardinality of an arbitrary family  of pairwise disjoint Borel measurable sets of positive $\lambda$-measure in $G$ is not more than countable.

\end{rem}
\begin{rem} Let consider a mapping $F: G \to  [0,1]$ defined by $F((x_k)_{k \in N})=\sum_{k \in N}\frac{x_k}{2^k}$  for $(x_k)_{k \in N} \in G$. This a Borel isomorphism between
 $G$ and $[0,1]$  such that following equality $\lambda(X)=\ell_1(F(X))$ holds true for each $X \in \mathcal{B}(G)$.  By virtue the latter relation, for each natural number $m$, the exists a partition $\{ X_k : 1 \le k \le m \}$ of the group $G$ into Haar ambivalent sets such that for each  $1 \le  i \le j \le m$ there is an isometric Borel measurable bijection $f_{(i,j)} : G \to G$ such that the set $f_{(i,j)}(X_i)\Delta X_j$ is shy, equivalently, of the $\lambda$-measure zero.
\end{rem}
By the scheme presented in the proof of the Theorem 6.21, one can get the validity of the following assertions.
\begin{thm} Let $\Theta_1$ be a subset of the $\Theta$ with $\hbox{card}(\Theta)\ge 2$. Then there exists an objective consistent estimate of an unknown parameter $\theta$ for the statistical structure $\{(G, \mathcal{B}(G), \mu_{\theta}): \theta \in \Theta_1 \}$ if and only if ~$\mbox{card}(\Theta_1) \le \aleph_0$ and $1/2 \notin \Theta_1$.
\end{thm}

\begin{thm} Let $\Theta_2$ be a subset of the $\Theta$ $\hbox{card}(\Theta)\ge 2$. Then there exists a strong  objective consistent estimate of an unknown parameter $\theta$ for the statistical
structure $\{(G, \mathcal{B}(G), \mu_{\theta}): \theta \in \Theta_2 \}$ if and only if ~$\hbox{card}(\Theta_2) < \aleph_0$ and $1/2 \notin \Theta_2$.
\end{thm}

%\section*{Acknowledgements} Thank you all for helping us writing this article

\end{document}